\documentclass[11pt]{amsart}
\usepackage{latexsym}
\usepackage{amssymb}
\usepackage{amsmath}

\newtheorem{theorem}{Theorem}[section]
\newtheorem{lemma}[theorem]{Lemma}
\newtheorem{proposition}[theorem]{Proposition}
\newtheorem{corollary}[theorem]{Corollary}
\newtheorem{definition}[theorem]{Definition}

\newtheorem{remark}[theorem]{Remark}

\newcommand\esssup{\mathop{\rm esssup}}

\newcommand\id{\mathop{\rm id}}

\newcommand\nph{\varphi}

\newcommand\ocl{\mathop{{\rm cl}_{\omega}}} 
\newcommand\ointer{\mathop{{\rm int}_{\omega}}} 
 
\newcommand\omax{\mathop{\rm OMAX}}
\newcommand\omin{\mathop{\rm OMIN}}
\newcommand\diag{\mathop{\rm diag}}

\newcommand{\cl}[1]{\mathcal{#1}}
\newcommand{\bb}[1]{\mathbb{#1}}

\begin{document}

\title[Operator system structures]{Operator system structures and extensions of Schur multipliers}

\author[Y.-F. Lin]{Ying-Fen Lin}
\address{Mathematical Sciences Research Centre, Queen's University Belfast, Belfast BT7 1NN, United Kingdom}
\email{y.lin@qub.ac.uk}

\author[I. G. Todorov]{Ivan G. Todorov}
\address{Mathematical Sciences Research Centre, Queen's University Belfast, Belfast BT7 1NN, United Kingdom, and 
School of Mathematical Sciences, Nankai University, 300071 Tianjin, China}
\email{i.todorov@qub.ac.uk}

\date{2 December 2018}

\maketitle

\begin{abstract}
For a given C*-algebra $\cl A$, we establish the existence of maximal and minimal operator $\cl A$-system structures on 
an AOU $\cl A$-space. In the case $\cl A$ is a W*-algebra, 
we provide an abstract characterisation of dual operator $\cl A$-systems, 
and study the maximal and minimal dual operator $\cl A$-system structures on 
a dual AOU $\cl A$-space. We introduce operator-valued Schur multipliers, and provide 
a Grothendieck-type characterisation. We study the positive extension problem for 
a partially defined operator-valued Schur multiplier $\nph$ and, 
under some richness conditions, characterise its 
affirmative solution in terms of the equality between the canonical and the maximal 
dual operator $\cl A$-system structures on an operator system naturally associated with the 
domain of $\nph$. 
\end{abstract}

%\tableofcontents

%%%%%%%%%%%%%%%%%%%%%%%%%%%%%%%%%%%%%%%%%%%%%%%%%%%%%%%%%%%%%%%%%%%
%%%%%%%%%%%%%%%%%%%%%%%%%%%%%%%%%%%%%%%%%%%%%%%%%%%%%%%%%%%%%%%%%%%%%%%%%%%%%%%%%%%%%%%%%%%%%%%%%%%%%%%%%%%%%%%%%%%%%%%%%%%%%%%%%%%%%%%%%%%%%%%%%%%%%%%%%%%%%%%%%%%%%%%%%%%%%%%%%%%%%%%%%%%%%%%%%%%%%%%%%

\section{Introduction}\label{s_intro}

The problem of completing a partially defined
matrix to a fully defined positive matrix 
has attracted considerable attention in the literature (see e.g. \cite{dg} and \cite{gjsw} and the references therein).  
Given an $n$ by $n$ matrix,
only a subset of whose entries are specified, this problem asks whether the remaining
entries can be determined so as to yield a positive matrix. 
For block operator matrices, this problem was considered in \cite{pps}, 
where the authors showed that it is closely related to questions about automatic complete positivity 
of certain positive linear maps. More specifically, 
one associates to the pattern $\kappa$ of the partially defined matrix (that is, the set of all given entries) 
the operator system $\cl S(\kappa)$ of all fully specified matrices supported by $\kappa$.
The positive completion problem is then linked to the question of whether the operator-valued 
Schur multiplier with domain $\cl S(\kappa)$ is completely positive. 

A continuous infinite dimensional version of the scalar-valued completion problem was considered in \cite{llt},
where the authors characterised the operator systems possessing the positive completion property 
in terms of an approximation of its positive cone via rank one operators.
The original motivation behind the present paper was
the study of the operator-valued, infinite dimensional and continuous, analogue of the 
positive completion problem. 
We relate the question to the automatic complete positivity 
of operator-valued Schur multipliers; 
in fact, we characterise the extendability of Schur multipliers 
in terms of an equality between operator system structures on an associated 
Archimedean order unit (AOU) *-vector space.

One of the fundamental representation theorems in Operator Space Theory is 
Choi-Effros Theorem \cite[Theorem 13.1]{Pa}, which characterises operator systems (that is, 
unital selfadjoint linear subspaces $\cl S$ of the space $\cl B(H)$ of all bounded linear operators on 
a Hilbert space $H$) abstractly, in terms of properties of the cones of positive elements in the $\cl S$-valued matrix space $M_n(\cl S)$. 
Operator $\cl A$-systems, that is, the operator systems 
which admit a bimodule action by a unital C*-algebra $\cl A$, can be characterised similarly in a way that
takes into account the extra $\cl A$-module structure \cite[Corollary 15.13]{Pa}. 
Dual operator systems -- that is, operator systems that are also dual operator spaces --
were characterised by D. P. Blecher and B. Magajna in \cite{bm}. However, no 
analogous representation of dual operator $\cl A$-systems, where $\cl A$ is a W*-algebra, 
has been known.

The idea of viewing operator spaces as a quantised version of Banach spaces 
has been very fruitful in Functional Analysis \cite{er}. Operator systems can in a similar vein be 
thought of as a quantised version of Archimedean order unit (AOU) *-vector spaces. 
The possible quantisations, or operator system structures, on 
a given AOU space, were first studied in \cite{ptt}, where it was shown 
that every AOU space possesses two extremal operator system structures. 
However, no similar development has been achieved for dual AOU spaces or for 
AOU $\cl A$-spaces. 

In this paper, we unify all aforementioned strands of questions. We provide a Choi-Effros type 
representation theorem for dual operator $\cl A$-systems. We study the operator $\cl A$-system structures 
on a given AOU $\cl A$-space, as well as the dual operator $\cl A$-system structures 
on a given dual AOU $\cl A$-space. The latter results are new even in the case where $\cl A$ coincides with the 
complex field. We introduce infinite dimensional measurable operator-valued Schur multipliers, and provide a characterisation 
that generalises their well-known description by A. Grothendieck \cite{Gro} 
in the scalar case (see also \cite{haag} and \cite{peller}).
Finally, we study the positive extension problem for operator-valued Schur multipliers, 
and characterise the possibility of such an extension by equality of the 
canonical and the maximal dual operator $\cl D$-system structures 
on the domain of the given Schur multiplier. 
Our context is that of an arbitrary (albeit standard) measure space $(X,\mu)$, which includes 
as a sub-case the discrete case and thus the finite case considered in \cite{pps}. 
In this context, the algebra $\cl D$ is the maximal abelian selfadjoint algebra corresponding to $L^{\infty}(X,\mu)$. 
Our results are a far reaching generalisation of the results of V. I. Paulsen, S. Power and R. R. Smith \cite{pps};
in particular, they provide a different view on the positive completion problem for block operator matrices
considered therein. 

The paper is organised as follows. After collecting some preliminaries in Section \ref{s_prel}, we 
establish, in Section \ref{s_eoss}, the existence of the 
minimal and the maximal operator $\cl A$-system structures on a AOU $\cl A$-space $V$, 
$\omin_{\cl A}(V)$ and $\omax_{\cl A}(V)$.
In case $V$ is a C*-algebra, $\omin_{\cl A}(V)$ was essentially defined in \cite{ps}, 
in relation with the problem of automatic complete positivity of $\cl A$-module maps,
whose completely bounded version was first considered by R. R. Smith in \cite{smith}
(see also the subsequent paper \cite{pss}). 
We show that 
$\omax_{\cl A}(V)$ (resp. $\omin_{\cl A}(V)$) is characterised by the automatic complete 
positivity of $\cl A$-bimodule positive maps from $V$ into any operator $\cl A$-system
(resp. from any operator $\cl A$-system into $V$). 

In Section \ref{s_doas}, we provide a characterisation theorem for dual operator $\cl A$-systems and,
in Section \ref{s_deoass}, we define dual AOU $\cl A$-spaces and undertake a development, 
analogous to the one in Section \ref{s_eoss}, for dual operator $\cl A$-system structures. 

In Section \ref{s_ism}, we introduce the operator-valued version of measurable Schur multipliers 
and provide a Grothendieck-type characterisation, noting the special case of positive
Schur multipliers.
In Section \ref{s_pe}, we study partially defined operator-valued Schur multipliers and 
their extension properties to a fully defined positive Schur multiplier. 
Associated with the domain $\kappa\subseteq X\times X$ of the Schur multiplier is an 
operator system $\cl S(\kappa)$. Our analysis depends on the presence of sufficiently many 
operators of finite rank in $\cl S(\kappa)$. 
We note that, of course, this holds true trivially in the classical matrix case. 
Under such richness conditions on the domain $\kappa$, we show that 
the positive extension problem for operator-valued Schur multipliers defined on $\kappa$ 
has an affirmative solution precisely when the canonical operator system structure of 
$\cl S(\kappa)$ coincides with its
maximal dual operator $\cl D$-system structure.

\medskip

%We finish this section by introducing some notation that will be used in the sequel. 
We denote by $(\cdot,\cdot)$ the inner product in a Hilbert space, and we use
$\langle \cdot, \cdot\rangle$ to designate duality paring.
We will assume some basic facts and notions from Operator Space Theory, for which we refer 
the reader to the monographs \cite{blm, er, Pa, pisier_intr}.

%%%%%%%%%%%%%%%%%%%%%%%%%%%%%%%%%%%%%%%%%%%%%%%%%%%%%%%%%%%%%%%%%%%
%%%%%%%%%%%%%%%%%%%%%%%%%%%%%%%%%%%%%%%%%%%%%%%%%%%%%%%%%%%%%%%%%%%%%%%%%%%%%%%%%%%%%%%%%%%%%%%%%%%%%%%%%%%%%%%%%%%%%%%%%%%%%%%%%%%%%%%%%%%%%%%%%%%%%%%%%%%%%%%%%%%%%%%%%%%%%%%%%%%%%%%%%%%%%%%%%%%%%%%%%

\section{Preliminaries}\label{s_prel}

In this section we recall basic results and introduce some new notions
that will be needed subsequently. 
If $W$ is a real vector space, a \emph{cone} in $W$ is a non-empty
subset $C \subseteq W$ with the following properties:

\medskip

(a) $\lambda v \in C$ whenever $\lambda \in \bb{R}^+ := [0,\infty)$
and $v \in C$;

(b) $v + w \in C$ whenever $v, w \in C$.

\medskip

\noindent A  \emph{*-vector space} is a complex vector space $V$
together with a map $^* : V \to V$ which is involutive (i.e. $(v^*)^*
= v$ for all $v \in V$) and conjugate linear (i.e. $(\lambda v + \mu w)^*
= \overline{\lambda} v^* + \overline{\mu} w^*$ for all $\lambda,\mu \in \bb C$ and all $v,w \in V$).
If $V$ is a *-vector space, then we let $V_h = \{x \in V :
x^* = x \}$ and call the elements of $V_h$ \emph{hermitian}. 
Note that $V_h$ is a real vector space.

An \emph{ordered *-vector space} \cite{pt} is a pair $(V, V^+)$ consisting of a 
*-vector space $V$ and a subset $V^+ \subseteq V_h$ satisfying
the following properties:

\medskip

(a) $V^+$ is a cone in $V_h$;

(b) $V^+ \cap -V^+ = \{ 0 \}$.

\medskip

Let $(V,V^+)$ be an ordered *-vector space.
We write $v \geq w$ or $w \leq v$ if $v,w\in V_h$ and $v - w \in V^+$.
Note that $v \in V^+$ if and only if $v \geq 0$; for this reason $V^+$ is 
referred to as the cone of \emph{positive} elements of $V$.

An element $e \in V_h$ is called an \emph{order unit} if
for every $v \in V_h$ there exists $r > 0$ such that $v\leq re$.
The order unit $e$ 
is called \emph{Archimedean} if, whenever
$v \in V$ and $re+v \in V^+$ for all $r >0$, we have that $v \in V^+$.
In this case, we call the triple $(V, V^+, e)$ an \emph{Archimedean
order unit *-vector space} (\emph{AOU space} for short).
Note that $(\bb{C}, \bb{R}^+,1)$ is an AOU space in a canonical fashion. 

Let $\cl A$ be a unital C*-algebra. 
Recall that a (complex) vector space $V$ is said to be an \emph{$\cl A$-bimodule} if 
it is equipped with bilinear maps $\cl A\times V\to V$, $(a,x)\to a\cdot x$ and 
$V\times \cl A\to V$, $(x,a)\to x\cdot a$, such that 
$(a\cdot x)\cdot b = a\cdot (x\cdot b)$, $(ab)\cdot x = a\cdot (b\cdot x)$, 
$x\cdot (ab)  = (x\cdot a)\cdot b$ and $1\cdot x = x$ 
for all $x\in V$ and all $a,b\in \cl A$.
If $V$ and $W$ are $\cl A$-bimodules, a linear map $\phi : V\to W$ is called 
an \emph{$\cl A$-bimodule map} if $\phi(a\cdot x\cdot b) = a\cdot \phi(x)\cdot b$, for all 
$x\in V$ and all $a,b\in \cl A$.

\begin{definition}\label{d_Asp}
Let $\cl A$ be a unital C*-algebra. An AOU space $(V,V^+,e)$ 
will be called an \emph{AOU $\cl A$-space} if $V$ is an 
$\cl A$-bimodule and the conditions 
\begin{equation}\label{eq_propm}
(a\cdot x)^* = x^*\cdot a^*, \ \ \ \  x\in V, a\in \cl A,
\end{equation}
\begin{equation}\label{eq_uni}
a\cdot e = e \cdot a, \ \ \ \  a\in \cl A,
\end{equation}
and
\begin{equation}\label{eq_aastar}
a^* \cdot x \cdot a \in V^+, \ \ \ \ x\in V^+, a\in \cl A, 
\end{equation}
are satisfied.
\end{definition}

For a complex vector space $V$, we let 
$M_{m,n}(V)$ denote the complex vector space of all $m$ by $n$ matrices with entries in $V$,
and often use the natural identification $M_{m,n}(V) \equiv M_{m,n} \otimes V$.
We write $A^t$ for the transpose of a matrix $A\in M_{m,n}(V)$.
We set $M_n(V) = M_{n,n}(V)$, $M_{m,n} = M_{m,n}(\bb{C})$ and $M_n = M_n(\bb{C})$; 
we write $I_n$ for the identity matrix in $M_n$. 
If $V$ is an AOU $\cl A$-space, we equip $M_n(V)$ with 
an involution by letting $(x_{i,j})^* = (x_{j,i}^*)$ and set
\begin{equation}\label{eq_matrixmod}
(a_{i,j})\cdot (x_{i,j}) = \left(\sum_{p = 1}^n a_{i,p}\cdot x_{p,j}\right)_{i,j} \ \mbox{ and } \ 
(x_{i,j})\cdot (b_{i,j}) = \left(\sum_{p = 1}^n x_{i,p}\cdot b_{p,j}\right)_{i,j},
\end{equation}
whenever $(x_{i,j})\in M_{m,n}(V)$, $(a_{i,j})\in M_{k,m}(\cl A)$ and 
$(b_{i,j})\in M_{n,l}(\cl A)$, $m,n$, $k,l$ $\in$ $\bb{N}$.

Let $\cl A$ be a unital C*-algebra and $(V,V^+,e)$ be an AOU $\cl A$-space. 
We write $e_n$ for the element of $M_n(V)$ whose diagonal entries coincide with $e$, 
while its off-diagonal entries are equal to zero. 
A family $(P_n)_{n\in \bb{N}}$, where $P_n\subseteq M_n(V)_h$ is a cone with $P_n \cap (-P_n) = \{0\}$, $n\in \bb{N}$, 
will be called a \emph{matrix ordering} of $V$. 
A matrix ordering $(P_n)_{n\in \bb{N}}$ will be called an 
\emph{operator $\cl A$-system structure} on $V$ if 
$P_1 = V^+$,
\begin{equation}\label{eq_diffs}
A^*\cdot X\cdot A \in P_n, \ \ \ \mbox{ whenever } X\in P_m \mbox{ and } A\in M_{m,n}(\cl A),
\end{equation}
and $e_n\in M_n(V)$ is an Archimedean order unit for $P_n$ for every $n\in \bb{N}$. 
%The sequence $(P_n)_{n\in \bb{N}}$ will be often referred to as the sequence of 
%\emph{matricial cones} of the operator $\cl A$-system structure,
Condition (\ref{eq_diffs}) will be referred to as the $\cl A$-compatibility of $(P_n)_{n\in \bb{N}}$. 
The triple $\cl S = (V,(P_n)_{n\in \bb{N}},e)$ is called an \emph{operator $\cl A$-system} (see \cite{Pa});
we write $M_n(\cl S)^+ = P_n$. 
Note that if $\cl B\subseteq \cl A$ is a unital C*-subalgebra, then 
every operator $\cl A$-system is also an operator $\cl B$-system in a canonical fashion. 
Operator $\bb{C}$-systems are called simply \emph{operator systems}.
We note that every operator system has a canonical operator space structure
(see \cite{Pa}). Note that condition (\ref{eq_uni}) is not a part of the standard definition of 
an operator $\cl A$-system; it is however automatically satisfied, as easily follows from Theorem \ref{th_repa} below. 

Let $H$ be a Hilbert space and $\cl B(H)$ be the space of all bounded linear operators on $H$.
We write $\cl B(H)^+$ for the cone of all positive operators in $\cl B(H)$. 
We identify $M_n(\cl B(H))$ with $\cl B(H^n)$, where $H^n$ denotes the 
direct sum of $n$ copies of $H$, and write $M_n(\cl B(H))^+ = \cl B(H^n)^+$, $n\in \bb{N}$.
It is straightforward to see that 
$\cl B(H)$ is an operator system when equipped with the adjoint operation as an involution, the
matrix ordering $(M_n(\cl B(H))^+)_{n\in \bb{N}}$, and the identity operator $I$ as an Archimedean matrix order unit.

Given AOU spaces $(V, V^+,e)$ and $(W, W^+,f)$, a linear map 
$\phi : V\to W$ is called \emph{unital} if $\phi(e) = f$, and \emph{positive} if 
$\phi(V^+)\subseteq W^+$. 
A linear map $s : V\to \bb{C}$ is called a \emph{state} on $V$ if $s$ is unital and positive.

Let $\cl S$ and $\cl T$ be operator systems with units $e$ and $f$, 
respectively.
For a linear map $\phi : \cl S\to \cl T$, 
we let $\phi^{(n,m)} : M_{n,m}(\cl S) \to M_{n,m}(\cl T)$ be the (linear) map given by 
$\phi^{(n,m)}((x_{i,j})_{i,j}) = (\phi(x_{i,j}))_{i,j}$, and set $\phi^{(n)} = \phi^{(n,n)}$. 
The map $\phi$ is called \emph{$n$-positive} if $\phi^{(n)}$ is positive, and 
it is called \emph{completely positive} if it is $n$-positive for all $n\in \bb{N}$. 
A bijective completely positive map $\phi : \cl S\to \cl T$ is called a 
\emph{complete order isomorphism} if its inverse $\phi^{-1}$ is completely positive. 
In this case, we call $\cl S$ and $\cl T$ are completely order isomorphic;
if $\phi$ is moreover unital, we say that $\cl S$ and $\cl T$ are unitally completely order isomorphic.
Further, $\phi$ is called a \emph{complete isometry} if $\phi^{(n)}$ is an isometry 
for each $n\in \bb{N}$. We note that a unital surjective map $\phi : \cl S\to \cl T$ is a complete isometry if and only if it is
a complete order isomorphism \cite[1.3.3]{blm}.

We refer the reader to \cite{Pa} for the general theory of operator systems and operator spaces, 
and in particular for the definition and basic properties of completely bounded maps.
The following characterisation, extending the well-known Choi-Effros representation theorem for 
operator systems \cite[Theorem 13.1]{Pa}, was established in \cite[Corollary 15.12]{Pa}.

\begin{theorem}\label{th_repa}
Let $\cl A$ be a unital C*-algebra and $\cl S$ be an operator system. 
The following are equivalent: 

(i) \ $\cl S$ is unitally completely order isomorphic to an operator $\cl A$-system;

(ii) there exist a Hilbert space $H$, a unital complete isometry $\gamma : \cl S\to \cl B(H)$
and a unital *-homomorphism $\pi : \cl A\to \cl B(H)$ such that 
$\gamma(a\cdot x) = \pi(a)\gamma(x)$ for all $x\in \cl S$ and all $a\in \cl A$. 
\end{theorem}

We note that, if $\cl A$ is a unital C*-algebra and 
$\cl S$ is an operator system that is also an operator $\cl A$-bimodule
satisfying (\ref{eq_propm}), then $\cl S$ is an operator $\cl A$-system precisely when
the family $(M_n(\cl S)^+)_{n\in \bb{N}}$ is $\cl A$-compatible.

%%%%%%%%%%%%%%%%%%%%%%%%%%%%%%%%%%%%%%%%%%%%%%%%%%%%%%%%%%%%%%%%%%%
%%%%%%%%%%%%%%%%%%%%%%%%%%%%%%%%%%%%%%%%%%%%%%%%%%%%%%%%%%%%%%%%%%%%%%%%%%%%%%%%%%%%%%%%%%%%%%%%%%%%%%%%%%%%%%%%%%%%%%%%%%%%%%%%%%%%%%%%%%%%%%%%%%%%%%%%%%%%%%%%%%%%%%%%%%%%%%%%%%%%%%%%%%%%%%%%%%%%%%%%%

\section{The extremal operator $\cl A$-system structures}\label{s_eoss}

In this section, we show that any AOU $\cl A$-space can be equipped with two extremal 
operator $\cl A$-system structures, and establish their universal properties. 
We first consider the minimal operator $\cl A$-system structure. Note that, in the case where 
the AOU $\cl A$-space is a C*-algebra containing $\cl A$, 
this operator system structure was first defined and studied in \cite{ps}. 

Let $\cl A$ be a unital C*-algebra and $(V,V^+,e)$ be an AOU $\cl A$-space. 
For $n\in \bb{N}$, let 
$$C_n^{\min}(V;\cl A) = \{X\in M_n(V)_h : C^*\cdot X\cdot C\in V^+, \mbox{ for all } C\in M_{n,1}(\cl A)\}.$$

\begin{remark}\label{r_scopmin}
{\rm Suppose that $(V,V^+,e)$ is an AOU $\cl A$-space and that $\cl B$ is a unital C*-subalgebra 
of $\cl A$. Then $(V,V^+,e)$ is also an AOU $\cl B$-space in the natural fashion. 
Clearly, $C_n^{\min}(V;\cl A) \subseteq C_n^{\min}(V;\cl B)$.
In particular, $C_n^{\min}(V;\cl A)$ is contained in 
$C_n^{\min}(V;\bb{C})$; note that the latter set coincides with the cone $C_n^{\min}(V)$ 
introduced in \cite[Definition 3.1]{ptt}.
}
\end{remark}

\begin{theorem}\label{th_cmina}
Let $\cl A$ be a unital C*-algebra and $(V,V^+,e)$ be an AOU $\cl A$-space.
Then $(C_n^{\min}(V;\cl A))_{n\in \bb{N}}$ is an operator $\cl A$-system structure on $V$. 
Moreover, if $(P_n)_{n\in \bb{N}}$ is an operator $\cl A$-system structure on $V$ then 
$P_n \subseteq C_n^{\min}(V;\cl A)$ for each $n\in \bb{N}$.
\end{theorem}
\begin{proof}
Since $V^+$ is a cone, $C_n^{\min}(V;\cl A)$ is a cone, too. 
As a consequence of \cite[Theorem 3.2]{ptt} and Remark \ref{r_scopmin}, 
$C_n^{\min}(V;\cl A)\cap (-C_n^{\min}(V;\cl A)) = \{0\}$.
If $X\in C_m^{\min}(V;\cl A)$, $A\in M_{m,n}(\cl A)$ and $C\in M_{n,1}(\cl A)$
then $AC\in M_{m,1}(\cl A)$ and hence 
$$C^*\cdot (A^*\cdot X\cdot A)\cdot C = (AC)^*\cdot X \cdot (AC) \in V^+,$$
showing that $A^*\cdot X\cdot A\in C_n^{\min}(V;\cl A)$.
Thus, the family $(C_n^{\min}(V;\cl A))_{n\in \bb{N}}$ is $\cl A$-compatible.

Suppose that $(P_n)_{n\in \bb{N}}$ is an operator $\cl A$-system structure on $V$. 
If $X\in P_n$ then, by $\cl A$-compatibility, $C^*\cdot X\cdot C\in P_1 = V^+$,
and hence $X\in C_n^{\min}(V;\cl A)$. Thus, $P_n\subseteq C_n^{\min}(V;\cl A)$. 
It will follow from the proof of Theorem \ref{th_cmaxa} below that $e_n$ is an order unit for $C_n^{\min}(V;\cl A)$.
To see that $e_n$ is Archimedean, suppose that $X + re_n\in C_n^{\min}(V;\cl A)$ 
for every $r > 0$. Let $C\in M_{n,1}(\cl A)$. Using (\ref{eq_uni}), we have 
$$C^*\cdot X\cdot C + r C^*C  \cdot e = C^*\cdot (X + r e_n)\cdot C \in V^+, \ \mbox{ for all } r > 0.$$
Let $\epsilon > 0$ and $T = (C^*C + \epsilon 1)^{-1/2}\in \cl A$. 
We have that 
$$C^*\cdot X\cdot C + r C^*C \cdot e + r\epsilon e \in V^+, \ \mbox{ for all } r > 0$$
and hence, by (\ref{eq_uni}) and (\ref{eq_aastar}), 
$$T(C^*\cdot X\cdot C)T + r e \in V^+, \ \mbox{ for all } r > 0.$$
Since $e$ is Archimedean for $V^+$, we have that 
$T(C^*\cdot X\cdot C)T \in V^+$. Applying (\ref{eq_aastar}) again, 
we conclude that 
$$C^*\cdot X\cdot C = T^{-1}(T(C^*\cdot X\cdot C)T)T^{-1} \in V^+;$$
thus $X\in C_n^{\min}(V;\cl A)$ and
the proof is complete.
\end{proof}

We call $(C_n^{\min}(V;\cl A))_{n\in \bb{N}}$ the \emph{minimal operator $\cl A$-system structure} on $V$, and let
$$\omin\mbox{}_{\cl A}(V) = \left(V, (C_n^{\min}(V;\cl A))_{n\in \bb{N}}, e\right).$$
The following theorem describes its universal property. Part (i) below 
was established in \cite{ps} in the case $V$ is a C*-algebra containing $\cl A$.

\begin{theorem}\label{th_umin}
Let $\cl A$ be a unital C*-algebra and $(V,V^+,e)$ be an AOU $\cl A$-space. 

(i) Suppose that $\cl S$ is an operator $\cl A$-system 
and $\phi : \cl S\to V$ is a positive $\cl A$-bimodule map. 
Then $\phi$ is completely positive as a map from $\cl S$ into $\omin\mbox{}_{\cl A}(V)$.

(ii) If $\cl T$ is an operator $\cl A$-system with underlying space $V$ and positive cone $V^+$, 
such that for every operator $\cl A$-system $\cl S$,
every positive $\cl A$-bimodule map $\phi : \cl S\to \cl T$ is completely positive, 
then there exists a unital $\cl A$-bimodule map $\psi : \cl T \to \omin_{\cl A}(V)$ that is a complete order isomorphism.
\end{theorem}
\begin{proof}
(i) 
Let $\cl S$ be an operator $\cl A$-system 
and $\phi : \cl S\to V$ be a positive $\cl A$-bimodule map.
Suppose that $X = (x_{i,j})\in M_n(\cl S)^+$ and $C = (a_i)_{i=1}^n\in M_{n,1}(\cl A)$. 
Then $C^*\cdot X \cdot C\in \cl S^+$; 
since $\phi$ is a positive $\cl A$-bimodule map, we have
\begin{eqnarray*}
C^*\cdot \phi^{(n)}(X)\cdot C & = & 
\sum_{i,j=1}^n a_i^* \cdot \phi(x_{i,j})\cdot a_j
= \phi\left(\sum_{i,j=1}^n a_i^* \cdot x_{i,j}\cdot a_j\right)\\
& = & 
\phi(C^*\cdot X \cdot C)\in V^+.
\end{eqnarray*}
Thus, $\phi^{(n)}$ maps $M_n(\cl S)^+$ into $C_n^{\min}(V;\cl A)$ and hence $\phi$ is completely positive.

(ii) Suppose that the operator $\cl A$-system $\cl T$ satisfies the properties in (ii). 
Since the identity $\id : \omin_{\cl A}(V)\to V$ is a positive $\cl A$-bimodule map, 
we have that $\id : \omin_{\cl A}(V) \to \cl T$ is completely positive. 
On the other hand, the identity $\id : \cl T\to V$ is also positive and $\cl A$-bimodular. 
By (i), $\id : \cl T \to \omin_{\cl A}(V)$ is completely positive, and we can take $\psi = \id$. 
\end{proof}

\medskip

We next consider the maximal operator $\cl A$-system structure. 
For $n\in \bb{N}$, set 
$$D_n^{\max}(V;\cl A) = \left\{\sum_{i=1}^k A_i^*\cdot x_i \cdot A_i : k\in \bb{N}, x_i\in V^+, 
A_i\in M_{1,n}(\cl A)\right\}$$
and let $\cl D^{\max}(V;\cl A) = (D_n^{\max}(V;\cl A))_{n\in \bb{N}}$.

\begin{remark}\label{r_scop}
{\rm 
Suppose that $(V,V^+,e)$ is an AOU $\cl A$-space and that $\cl B$ is a unital C*-subalgebra of $\cl A$. 
Clearly, $D_n^{\max}(V;\cl B) \subseteq D_n^{\max}(V;\cl A)$.
Given any AOU space $(V,V^+,e)$, in \cite{ptt} the authors defined 
$$D_n^{\max}(V) = \left\{\sum_{i=1}^k B_i \otimes x_i : k\in \bb{N}, x_i\in V^+, B_i\in M_{n}^+\right\}.$$ 
Since every matrix $B\in M_n^+$ is the sum of matrices of the form $A^*A$, where $A\in M_{1,n}$, 
we have that $D_n^{\max}(V) = D_n^{\max}(V;\bb{C}1)$.}
\end{remark}

\begin{lemma}\label{min}
Let $\cl A$ be a unital C*-algebra and $(V,V^+,e)$ be an AOU $\cl A$-space.
Let $P_n\subseteq M_n(V)_h$ be a cone, $n\in\bb{N}$, such that the family
$(P_n)_{n=1}^{\infty}$ is $\cl A$-compatible and $P_1 = V^+$.
Then $D_n^{\max}(V;\cl A)\subseteq P_n$, for each $n\in \bb{N}$.
\end{lemma}
\begin{proof}
Let $n\in \bb{N}$.
%Suppose that $(P_n)_{n=1}^{\infty}$ is an $\cl A$-compatible collection of
%cones with $P_1 = V^+$. 
If $A\in M_{1,n}(\cl A)$ then
$$A^* \cdot V^+ \cdot A = A^* \cdot P_1 \cdot A \subseteq P_n.$$ 
Thus $D_n^{\max}(V; \cl A) \subseteq P_n$.
\end{proof}

If $x_1,\dots,x_n\in V$ we let
$\diag(x_1,\dots,x_n)$
denote the element of $M_n(V)$ with $x_1,\dots,x_n$ on its diagonal
(in this order) and zeros elsewhere.

\begin{proposition}\label{minmax}
Let $\cl A$ be a unital C*-algebra and $(V,V^+,e)$ be an AOU $\cl A$-space. The
following hold:
%\begin{itemize}

(i) \  $D_n^{\max}(V;\cl A) = \{A^* \cdot \diag(x_1,\dots,x_m) \cdot  A :
A\in M_{m,n}(\cl A), x_i\in V^+, i= 1,\dots,m,$ $m\in\bb{N}\}$;

(ii) $\cl D^{\max}(V;\cl A)$ is an $\cl A$-compatible 
matrix ordering on $V$ and $e$ is
a matrix order unit for it.
%\end{itemize}
\end{proposition}
\begin{proof} 
(i) Let $D_n$ denote the right hand side of the equality in (i). We first
observe that $D_n$ is a cone in $M_n(V)_h$. If $x_1,\dots,x_m\in
V^+$ and $A = (a_{i,k})_{i,k}\in M_{m,n}(\cl A)$ then the
$(i,j)$-entry of $A^* \cdot \diag(x_1,\dots,x_m)\cdot A$ is equal to
$\sum_{k=1}^m a_{k,i}^*\cdot x_k \cdot a_{k,j}$ and, by (\ref{eq_propm}),
$$\left(\sum_{k=1}^m a_{k,i}^*\cdot x_k \cdot a_{k,j}\right)^* = 
\sum_{k=1}^m a_{k,j}^*\cdot x_k \cdot a_{k,i};$$
thus, $D_n\subseteq M_n(V)_h$. It is clear that $D_n$ is closed
under taking multiples with non-negative real numbers.
Fix elements
$$A^* \cdot \diag(x_1,\dots,x_m)\cdot A, \ \mbox{ and } \ 
B^* \cdot \diag(y_1,\dots,y_k)\cdot B$$
of $D_n$. Letting $C = [A \ B]^t$, we have 
\begin{eqnarray*}
& & A^*\cdot \diag(x_1,\dots,x_m)\cdot A + 
B^*\cdot \diag(y_1,\dots,y_k)\cdot B\\ & = & 
C^*\cdot \diag(x_1,\dots,x_m,y_1,\dots,y_k) \cdot C\in D_n \ ;
\end{eqnarray*}
in other words, $D_n$ is a cone. 
If $B\in M_{n,l}(\cl A)$ then
$$B^*\cdot (A^*\cdot \diag(x_1,\dots,x_m)\cdot A)\cdot B =
(AB)^*\cdot \diag(x_1,\dots,x_m)\cdot (AB)\in D_l,$$ 
and so
$(D_n)_{n=1}^{\infty}$ is $\cl A$-compatible. 
By (\ref{eq_aastar}), $D_1 = V^+$. 
Lemma \ref{min} now implies that $D_n^{\max}(V;\cl A)\subseteq D_n$ for $n\in \bb{N}$.

On the other hand, if $x_1,\dots,x_m\in V^+$ then, letting 
$E_{i}\in M_{1,m}(\cl A)$ be the row with $1$ at the $i$th coordinate and zeros elsewhere, 
we have that
$$\diag(x_1,\dots,x_m) = \sum_{i=1}^m E_{i}^*\cdot x_i\cdot E_{i}\in D_m^{\max}(V; \cl A).$$
Since the family $\cl D^{\max}(V;\cl A)$ is $\cl A$-compatible,
$$A^*\cdot \diag(x_1,\dots,x_m)\cdot A \in D_n^{\max}(V;\cl A), \ \ \ A \in M_{m,n}(\cl A).$$
Thus, $D_n\subseteq D_n^{\max}(V;\cl A)$ and (i) is established.

(ii) By Remark \ref{r_scop} and \cite[Proposition 3.10]{ptt}, 
$e_n$ is an order unit for $D_n^{\max}(V;\bb{C}1)$. 
By Remark \ref{r_scop} again, $e_n$ is an order unit for $D_n^{\max}(V;\cl A)$.
\end{proof}

For $n\in \bb{N}$, let 
$$C_n^{\max}(V;\cl A) = \{X\in M_n(V) : X + re_n \in D_n^{\max}(V;\cl A) \mbox{ for every } r> 0\}.$$

\begin{theorem}\label{th_cmaxa}
Let $\cl A$ be a unital C*-algebra and $(V,V^+,e)$ be an AOU $\cl A$-space.
Then $(C_n^{\max}(V;\cl A))_{n\in \bb{N}}$ is an operator $\cl A$-system structure on $V$. 
Moreover, if $(P_n)_{n\in \bb{N}}$ is an operator $\cl A$-system structure on $V$ then 
$$C_n^{\max}(V;\cl A)\subseteq P_n$$ 
for each $n\in \bb{N}$.
\end{theorem}
\begin{proof}
Write $C_n = C_n^{\max}(V;\cl A)$, $n\in \bb{N}$. 
By Theorem \ref{th_cmina} and 
Lemma \ref{min}, $C_n\subseteq C_n^{\min}(V;\cl A)$; thus, $C_n \cap (-C_n) = \{0\}$. 
Since $e_n$ is an order unit for $D_n^{\max}(V;\cl A)$ and $D_n^{\max}(V;\cl A)\subseteq C_n$, 
we have that $e_n$ is an order unit for $C_n$.

Suppose that $X\in M_n(V)_h$ is such that $X + re_n\in C_n$ for every $r > 0$. 
Let $\epsilon > 0$; then 
$$X + \epsilon e_n = \left(X + \frac{\epsilon}{2} e_n\right) + \frac{\epsilon}{2} e_n \in D_n^{\max}(V;\cl A)$$
and hence $X\in C_n$. 
Thus, $e_n$ is an Archimedean matrix order unit for $C_n$. 

It remains to show that the family $(C_n)_{n\in \bb{N}}$ is $\cl A$-compatible. 
To this end, let $X\in C_n$ for some $n\in \bb{N}$ and $A\in M_{n,m}(\cl A)$. 
By Proposition \ref{minmax}, 
there exists $R > 0$ such that 
$$R e_m - A^*\cdot e_n \cdot A \in D_m^{\max}(V;\cl A).$$
Let $r > 0$. 
Since $X + \frac{r}{R} e_n\in D_n^{\max}(V;\cl A)$ and 
the family $\cl D^{\max}(V;\cl A)$ is $\cl A$-compatible
(Proposition \ref{minmax}), we have 
\begin{eqnarray*}
& & 
A^*\cdot X\cdot A + re_m \\
& = & 
\left(A^*\cdot \left(X + \frac{r}{R} e_n\right)\cdot A\right) + 
r\left(e_m - \frac{1}{R} A^*\cdot e_n \cdot A\right) \in D_m^{\max}(V;\cl A).
\end{eqnarray*}  
It follows that $A^*\cdot X\cdot A\in C_m$. 
Thus, $(C_n)_{n\in \bb{N}}$ is an operator $\cl A$-system structure on $V$.

Suppose that $(P_n)_{n\in \bb{N}}$ is an operator $\cl A$-system structure on $V$
and $X\in C_n$ for some $n\in \bb{N}$. 
By Lemma \ref{min}, $X + r e_n\in P_n$ for all $r > 0$ and since 
$e_n$ is an Archimedean order unit for $P_n$, we conclude that $X\in P_n$. 
Thus, $C_n\subseteq P_n$, and the proof is complete. 
\end{proof}

We call
$(C_n^{\max}(V;\cl A))_{n\in \bb{N}}$ \emph{the maximal operator $\cl A$-system structure} on $V$
and let 
$$\omax\mbox{}_{\cl A}(V) = (V,(C_n^{\max}(V;\cl A))_{n\in \bb{N}},e).$$

\medskip

\noindent {\bf Remark. } Recall that, given an AOU space $(V,V^+,e)$, the 
\emph{maximal operator system structure} $(C_n^{\max}(V))_{n\in \bb{N}}$ on $V$
was defined in \cite{ptt} by letting $C_n^{\max}(V)$ be the Archimedeanisation 
of the cone $D_n^{\max}(V)$ defined in Remark \ref{r_scop}.
It follows that the maximal operator system 
$\omax(V)$ defined in \cite{ptt} coincides with $\omax_{\bb{C}}(V)$.

\begin{theorem}\label{th_acp}
Let $\cl A$ be a unital C*-algebra and $(V,V^+,e)$ be an AOU $\cl A$-space. 

(i) Suppose that $\cl S$ is an operator $\cl A$-system 
and $\phi : V\to \cl S$ is a positive $\cl A$-bimodule map. 
Then $\phi$ is completely positive as a map from $\omax\mbox{}_{\cl A}(V)$ into $\cl S$.

(ii) Suppose that $\cl T$ is an operator $\cl A$-system with underlying space $V$ and positive cone $V^+$, 
such that for every operator $\cl A$-system $\cl S$,
every positive $\cl A$-bimodule map $\phi : \cl T\to \cl S$ is completely positive.
Then there exists a unital $\cl A$-bimodule map $\psi : \cl T\to \omax_{\cl A}(V)$ that is a complete order isomorphism.
\end{theorem}

\begin{proof}
(i) Let $\cl S$ is an operator $\cl A$-system and $\phi : V\to \cl S$ be a positive $\cl A$-bimodule map.
The modularity property of $\phi$ and the definition of $D_n^{\max}(V;\cl A)$ imply that 
$\phi^{(n)}(D_n^{\max}(V;\cl A)) \subseteq M_n(\cl S)^+$. 
Suppose that $X\in C_n^{\max}(V;\cl A)$. Letting $z = \phi(e)$, we now have that 
$\phi^{(n)}(X) + r(z\otimes I_n) \in M_n(\cl S)^+$ for every $r > 0$. Since $M_n(\cl S)^+$ is closed, this implies that 
$\phi^{(n)}(X) \in M_n(\cl S)^+$. 
Thus,  $\phi$ is completely positive. 

(ii) is similar to the proof of Theorem \ref{th_umin} (ii). 
\end{proof}

\noindent {\bf Remark. } 
Let $\cl A$ be a C*-algebra and 
$\frak{A}_{\cl A}$ (resp. $\frak{S}_{\cl A}$) be the category, whose objects are AOU $\cl A$-spaces
(resp. operator $\cl A$-systems) and whose morphisms are unital positive
(resp. unital completely positive) maps. 
It is easy to see that the 
correspondences $V\to \omin_{\cl A}(V)$ and $V\to \omax_{\cl A}(V)$
are covariant functors from $\frak{A}_{\cl A}$ into $\frak{S}_{\cl A}$.

\medskip

We finish this section with considering the case where $V = M_k$ and $\cl A$ coincides with 
its subalgebra $\cl D_k$ of all diagonal matrices.

\begin{proposition}\label{p_mn}
We have that $M_k = \omin_{\cl D_k}(M_k) = \omax_{\cl D_k}(M_k)$.
\end{proposition}
\begin{proof}
Suppose that $X = (X_{i,j})_{i,j}$ belongs to $M_n(\omin_{\cl D_k}(M_k))^+$.
Let $\xi = (\lambda_{i,1},\dots,\lambda_{i,k})_{i=1}^n$ be a vector in $\bb{C}^{nk}$. 
Let $D_i = \diag(\lambda_{i,1},\dots,\lambda_{i,k})$, 
and write $\xi_i$ for the vector $(\lambda_{i,1},\dots,\lambda_{i,k})$ in $\bb{C}^k$, $i = 1,\dots,n$.
Letting $e$ be the vector in $\bb{C}^k$ with all entries equal to one, we have 
$$(X\xi,\xi) = \sum_{i,j=1}^n (X_{i,j}\xi_j,\xi_i) = \sum_{i,j=1}^n (D_i^* X_{i,j} D_je,e).$$
It follows by the assumption that $(X\xi,\xi)\geq 0$; thus, $X\in M_{nk}^+$ 
and, by Theorem \ref{th_cmina},  $M_k = \omin_{\cl D_k}(M_k)$.

Now fix $X = (X_{i,j})_{i,j} \in M_{nk}^+$.
Since $X$ is the sum of rank one operators in $M_{nk}^+$, in order to show that $X\in M_n(\omax_{\cl D_k}(M_k))^+$, 
it suffices to assume that $X$ is itself of rank one. 
Write $X = R R^*$, where $R \in M_{nk,1}$, and suppose that $R = (R_1,\dots,R_n)^t$, where $R_i\in M_{k,1}$, $i = 1,\dots,n$. 
We have that $X = (R_iR_j^*)_{i,j=1}^n$. Let 
$J\in M_k$ be the matrix with all its entries equal to one, and let 
$D_i$ be the diagonal matrix whose entries coincides with the vector $R_i$, $i = 1,\dots,n$.
Then $X = (D_i J D_j^*)_{i,j=1}^n$, showing that $X\in M_n(\omax_{\cl D_k}(M_k))^+$.
By Theorem \ref{th_cmaxa}, $M_k = \omax_{\cl D_k}(M_k)$.
\end{proof}

\noindent {\bf Remark. } 
We note that the minimal and the maximal operator $\cl A$-system structure are in general distinct.
Indeed, this is the case even when $V = M_k$ and $\cl A = \bb{C}I$ \cite{ptt}.

%%%%%%%%%%%%%%%%%%%%%%%%%%%%%%%%%%%%%%%%%%%%%%%%%%%%%%%%%%%%%%%%%%%
%%%%%%%%%%%%%%%%%%%%%%%%%%%%%%%%%%%%%%%%%%%%%%%%%%%%%%%%%%%%%%%%%%%%%%%%%%%%%%%%%%%%%%%%%%%%%%%%%%%%%%%%%%%%%%%%%%%%%%%%%%%%%%%%%%%%%%%%%%%%%%%%%%%%%%%%%%%%%%%%%%%%%%%%%%%%%%%%%%%%%%%%%%%%%%%%%%%%%%%%%

\section{Dual operator $\cl A$-systems}\label{s_doas}

In this section, we establish a representation theorem for dual operator $\cl A$-systems. 
An operator system $\cl S$
%Recall that $\cl S$ is then an operator space in a canonical fashion \cite{Pa}.
is called a \emph{dual operator system} if it is a dual operator space, 
that is, if there exists an operator space $\cl S_*$ such that 
$(\cl S_*)^* \cong \cl S$ completely isometrically \cite{bm}. 
Here, and in the sequel, we denote by $\cl X^*$ the operator space dual \cite{blm} of an operator space $\cl X$, 
and we use the same notation for the dual Banach space of a normed space $\cl X$; 
it will be clear from the context with which category we are working.

Let $\cl S$ be an operator system. 
If $H$ is a Hilbert space and $\phi : \cl S\to \cl B(H)$ 
is a unital complete isometry such that $\phi(\cl S)$ is weak* closed, 
then $\phi(\cl S)$, and therefore $\cl S$, is a dual operator space;
thus, in this case, $\cl S$ is a dual operator system. 
The converse statement was established by Blecher and Magajna in \cite{bm}.

\begin{theorem}[\cite{bm}]\label{th_bm}
If $\cl S$ is a dual operator system then there exists a 
Hilbert space $H$, a weak* closed operator system $\cl U\subseteq \cl B(H)$ and 
a unital surjective complete order isomorhism  
$\phi : \cl S\to \cl U$ that is also a a weak* homeomorphism.
\end{theorem}

\begin{remark}\label{r_sdu}
{\rm Suppose that $\cl S$ is a dual operator system and $\cl S_*$ is an 
operator space such that, up to a complete isometry, $\cl S = (\cl S_*)^*$.
Then $M_n(\cl S)$ is an operator system in a canonical fashion; 
in fact, if $\cl S\subseteq \cl B(H)$ for some Hilbert space $H$, then $M_n(\cl S)\subseteq \cl B(H^n)$. 
By \cite[1.6.2]{blm}, up to a complete isometry, $M_n(\cl S) = (\cl S_*\hat{\otimes} M_n^*)^*$,
where $\hat{\otimes}$ is the projective operator space tensor product. 
It follows that $M_n(\cl S)$ is a dual operator system, 
and its canonical weak* topology coincides with the topology of entry-wise weak* convergence:
for a net $((x_{i,j}^{\alpha})_{i,j})_{\alpha}\subseteq M_n(\cl S)$ and an element $(x_{i,j})_{i,j} \in M_n(\cl S)$, 
we have 
$$\left((x_{i,j}^{\alpha}\right)_{i,j})_{\alpha}\to^{w^*}_{\alpha} \left(x_{i,j}\right)_{i,j}
\ \Longleftrightarrow \ \left\langle x_{i,j}^{\alpha},\phi \right\rangle \to\mbox{}_{\alpha} \left\langle x_{i,j},\phi\right\rangle, 
\ i,j = 1,\dots,n, \ \phi\in \cl S_*.$$}
\end{remark}

Recall that a \emph{W*-algebra} is a C*-algebra that is also a dual Banach space; 
by Sakai's Theorem \cite{sakai}, every W*-algebra possesses a faithful *-representation 
on a Hilbert space $H$, whose image is a von Neumann algebra 
(that is, a weak* closed subalgebra of $\cl B(H)$ containing the identity operator), 
which is also a weak* homeomorphism.
%Theorem \ref{th_bm} can be viewed as a version of the latter result for operator systems.

\begin{definition}\label{d_dasys}
Let $\cl A$ be a W*-algebra. An operator system $\cl S$
will be called a \emph{dual operator $\cl A$-system} if 
\begin{itemize}
\item[(i)]  $\cl S$ is an operator $\cl A$-system, 
\item[(ii)] $\cl S$ is a dual operator system, and 
\item[(iii)] the map from $\cl A\times \cl S$ into $\cl S$, sending the pair $(a,x)$ to $a\cdot x$,
is separately weak* continuous. 
\end{itemize}
\end{definition}

Note that, if $\cl S$ is a dual operator system then the involution is weak* continuous, 
and thus (\ref{eq_propm}) implies that if $\cl S$ is in addition a dual operator $\cl A$-system then 
the map
$$\cl A\times \cl S\times \cl A\to \cl S, \ \ \ (a,x,b)\to a\cdot x\cdot b,$$
is separately weak* continuous.

If $\cl S$ and $\cl T$ are dual operator systems, a linear map $\phi : \cl S\to \cl T$ 
will be called \emph{normal} if it is weak* continuous. 
Suppose that $H$ is a Hilbert space, 
$\gamma : \cl S\to \cl B(H)$
is a unital complete order isomorphism such that $\gamma(\cl S)$ is weak* closed and 
$\gamma : \cl S\to \gamma(\cl S)$ is a weak* homeomorphism, and
$\pi : \cl A\to \cl B(H)$ is a unital normal *-homomorphism such that 
$\gamma(a\cdot x) = \pi(a)\gamma(x)$ for all $x\in \cl S$ and all $a\in \cl A$. 
It is clear that, in this case, $\cl S$ is a dual operator $\cl A$-system. 
Theorem \ref{th_repdoas} below establishes the converse of this fact.
The result is both a weak* version of Theorem \ref{th_repa} 
and an $\cl A$-module version of Theorem \ref{th_bm}. 

We will need two lemmas. Recall that, if $\cl A$ is a W*-algebra and $n\in \bb{N}$
then $M_n(\cl A)$ is a W*-algebra in a canonical way.

\begin{remark}\label{l_matrdops}
{\rm Let $\cl A$ be a W*-algebra and $\cl S$ be a dual operator $\cl A$-system.
It is straightforward to verify that $M_n(\cl S)$ is a dual operator $M_n(\cl A)$-system,
when it is equipped with the action defined in (\ref{eq_matrixmod}). }
\end{remark}

%\begin{proof}
%Conditions (i) and (ii) from Definition \ref{d_dasys} are straightforward. 
%The separate weak* continuity of the module operation $M_n(\cl A) \times M_n(\cl S)\to M_n(\cl S)$
%follows from Remark \ref{r_sdu} and the fact that the module operation $\cl A\times\cl S\to \cl S$ is weak* continuous.
%\end{proof}

\begin{lemma}\label{l_funct}
Let $\cl A$ be a W*-algebra, $\cl S$ be a dual operator $\cl A$-system
and $\phi : \cl S\to \bb{C}$ be a normal state. 
Then the functional $\omega : \cl A\to \bb{C}$ given by 
$\omega(a) = \phi(a\cdot 1)$, $a\in \cl A$, is a normal state of $\cl A$ and 
\begin{equation}\label{eq_ineqq}
|\phi(a\cdot x\cdot b)|\leq \omega(aa^*)^{1/2}\omega(b^*b)^{1/2},
\end{equation}
for all $a\in M_{1,m}(\cl A)$, $b\in M_{m,1}(\cl A)$, $x\in M_m(\cl S)$ with $\|x\|\leq 1$, and
$m\in \bb{N}$.
\end{lemma}

\begin{proof}
Let $H$, $\gamma$ and $\pi$ be as in Theorem \ref{th_repa}, and 
let $\phi' : \gamma(\cl S)\to \bb{C}$ be given by $\phi'(\gamma(x)) = \phi(x)$, $x\in \cl S$. 
If $a,b\in \cl A$ then
\begin{eqnarray*}
\omega(ab) 
& = & \phi((ab)\cdot 1) = \phi'(\gamma((ab)\cdot 1)) = 
\phi'(\pi(ab)\gamma(1)) = \phi'(\pi(ab))\\
& = & 
\phi'(\pi(a)\gamma(1)\pi(b)) = \phi'(\gamma(a\cdot 1\cdot b)) = \phi(a\cdot 1\cdot b).
\end{eqnarray*}
Thus, 
$\omega(a^*a) = \phi(a^*\cdot 1\cdot a)\geq 0$ for every $a\in \cl A$, and hence 
$\omega$ is positive. 
Moreover, $\omega(1) = \phi(1) = 1$ and hence $\omega$ is a state. 
By the separate weak* continuity of the $\cl A$-module action on $\cl S$, the state $\omega$ is normal.

Suppose that $\phi'$ has the form 
$$\phi'(T) = \sum_{i=1}^{\infty} (T\xi_i,\xi_i), \ \ \ T\in \gamma(\cl S),$$
where $(\xi_i)_{i\in \bb{N}}\subseteq H$ with $\sum_{i=1}^{\infty}\|\xi_i\|^2  = 1$.
If $x\in M_m(\cl S)$, $\|x\|\leq 1$,  
$a\in M_{1,m}(\cl A)$ and $b\in M_{m,1}(\cl A)$, then 
\begin{eqnarray*}
|\phi(a\cdot x\cdot b)|
& = & 
\left|\phi' \left(\pi^{(1,m)}(a)\gamma^{(m)}(x)\pi^{(m,1)}(b)\right)\right| \\
& = & 
\left|\sum_{i=1}^{\infty} \left(\pi^{(1,m)}(a)\gamma^{(m)}(x)\pi^{(m,1)}(b)\xi_i,\xi_i\right)\right|\\
& \leq & 
\sum_{i=1}^{\infty} \left|\left(\gamma^{(m)}(x)\pi^{(m,1)}(b)\xi_i,\pi^{(m,1)}(a^*)\xi_i\right)\right|\\
& \leq & 
\left(\sum_{i=1}^{\infty} \left\|\pi^{(m,1)}(b)\xi_i\right\|^2\right)^{1/2} 
\left(\sum_{i=1}^{\infty} \left\|\pi^{(m,1)}(a^*)\xi_i\right\|^2\right)^{1/2}\\
& = & 
\phi'(\pi(b^*b))^{1/2} \phi'(\pi(aa^*))^{1/2} 
= 
\omega(aa^*)^{1/2}\omega(b^*b)^{1/2}.
\end{eqnarray*}
\end{proof}

We will need the following modification of a result of 
R. R. Smith \cite{smith} on automatic complete boundedness. 
Its proof is a straightforward modification of the proof of \cite[Theorem 2.1]{smith} and is hence omitted.

\begin{theorem}\label{l_sm}
Let $\cl A$ be a unital C*-algebra, $\cl S$ be an operator $\cl A$-system and 
$\rho : \cl A\to \cl B(H)$ be a cyclic *-representation.
Suppose that $\Phi : \cl S\to \cl B(H)$ is a linear map such that 
$\Phi(a\cdot x\cdot b) = \rho(a)\Phi(x)\rho(b)$ for all $x\in \cl S$ and all $a,b\in \cl A$. 
%\begin{itemize}
%\item[(i)] If $\Phi$ is positive then $\Phi$ is completely positive;
%\item[(ii)] 
If $\Phi$ is contractive then $\Phi$ is completely contractive. 
%\end{itemize}
\end{theorem}

\begin{theorem}\label{th_repdoas}
Let $\cl A$ be a W*-algebra and $\cl S$ be a dual operator $\cl A$-system. 
Then there exist a Hilbert space $H$, a unital complete order embedding  
$\gamma : \cl S\to \cl B(H)$ with the property that $\gamma(\cl S)$ is weak* closed and $\gamma$ is a  weak* homeomorphism,
and a unital normal *-homomorphism $\pi : \cl A\to \cl B(H)$, such that 
\begin{equation}\label{eq_modc}
\gamma(a\cdot x) = \pi(a)\gamma(x), \ \ \ x\in \cl S, a\in \cl A.
\end{equation}
\end{theorem}
\begin{proof}
The proof is motivated by the proof of \cite[Theorem 1.1]{bm} and 
relies on ideas which go back to the proof of Ruan's Theorem 
\cite[Theorem 2.3.5]{er}. 
Fix $n\in \bb{N}$ and let $\cl B = M_n(\cl A)$. 
By Remark \ref{l_matrdops}, $M_n(\cl S)$ is a dual operator $\cl B$-system.
Let $x\in M_n(\cl S)$ be a selfadjoint element of norm one and $\epsilon \in (0,1)$.
By the proof of Theorem 1.1 given in \cite{bm}, there exists 
a normal state $\phi$ on $M_n(\cl S)$ such that
\begin{equation}\label{eq_>ep}
|\phi(x)| > 1 - \epsilon. 
\end{equation}
Let $\omega : \cl B\to \bb{C}$ be the normal state given by 
$\omega(b) = \phi(b\cdot 1)$, $b\in \cl B$. 
By Lemma \ref{l_funct},
\begin{equation}\label{eq_om}
|\phi(a\cdot y \cdot b)|\leq \omega(aa^*)^{1/2}  \omega(b^*b)^{1/2},
\end{equation}
for all $y\in M_{nm}(\cl S)$ with $\|y\|\leq 1$, 
$a\in M_{1,m}(\cl B)$ and $b\in M_{m,1}(\cl B)$, $m\in \bb{N}$.

Let $\rho : \cl B\to \cl B(H)$ be the GNS representation arising from $\omega$
and $\xi$ be its corresponding unit cyclic vector. 
By \cite[Proposition III.3.12]{t}, $\rho$ is normal. 
It follows that there exists a normal unital *-representation $\theta : \cl A\to \cl B(K)$ 
such that, up to unitary equivalence, $H = K\otimes \bb{C}^n$ and 
$\rho = \theta^{(n)}$. 
Inequality (\ref{eq_om}) implies 
$$|\phi(a^*\cdot y \cdot b)|\leq \|\rho(b)\xi\|\|\rho(a)\xi\| \|y\|, \ \ a,b\in \cl B, y\in M_n(\cl S).$$
Thus, the sesqui-linear form $L_y : (\rho(\cl B)\xi)\times (\rho(\cl B)\xi) \to \bb{C}$
given by 
$$L_y(\rho(b)\xi,\rho(a)\xi) =  \phi(a^*\cdot y \cdot b), \ \ \  a,b\in \cl B,$$
is bounded and has norm not exceeding $\|y\|$.
It follows that there exists a linear operator
$\Phi(y) : \rho(\cl B)\xi\to \rho(\cl B)\xi$ such that
\begin{equation}\label{ope}
(\Phi(y)\rho(b)\xi,\rho(a)\xi) =  \phi(a^*\cdot y\cdot b), \ \ \ a,b\in \cl B,
\end{equation}
and 
\begin{equation}\label{eq_Phiy}
\|\Phi(y)\|\leq \|y\|.
\end{equation}
Since  $\rho(\cl B)\xi$ in dense in $H$, the operator $\Phi(y)$ can be extended 
to an operator on $H$. 
By (\ref{ope}), the map $\Phi : M_n(\cl S)\to \cl B(H)$ is linear and hermitian and, by (\ref{eq_Phiy}), it is contractive.

For $a,b,c,d\in \cl B$, by (\ref{ope}), we have 
$$(\Phi(c^* \cdot y \cdot d)\rho(b)\xi,\rho(a)\xi) 
=  (\rho(c^*)\Phi(y)\rho(d)\rho(b)\xi,\rho(a)\xi).$$
The density of $\rho(\cl B)\xi$ in $H$ now implies that
\begin{equation}\label{eq_mod}
\Phi(c^* \cdot y \cdot d) = \rho(c^*)\Phi(y)\rho(d), \ \ \ c,d\in \cl B, y\in M_n(\cl S).
\end{equation}

We show that $\Phi$ is weak* continuous. Suppose that 
$(y_{\alpha})_{\alpha}\subseteq M_n(\cl S)$ is a net of contractions such that 
$y_{\alpha}\to_{\alpha} y$ in the weak* topology, for some $y\in M_n(\cl S)$. 
Fix $\delta > 0$, $\eta,\zeta\in H$, and choose $a,b\in \cl B$ such that 
$$\|\rho(b)\xi - \eta\| < \delta \  \mbox{ and } \ \|\rho(a)\xi - \zeta\| < \delta.$$
Let ${\alpha}_0$ be such that $|\phi(a^* \cdot y_{\alpha} \cdot b) - \phi(a^* \cdot y \cdot b)| < \delta$ if ${\alpha}\geq {\alpha}_0$. 
For $\alpha\geq \alpha_0$ we have
\begin{eqnarray*}
& & 
|(\Phi(y_{\alpha})\eta,\zeta) - (\Phi(y)\eta,\zeta)|\\ 
& \leq & 
|(\Phi(y_{\alpha})\eta,\zeta) - (\Phi(y_{\alpha})\rho(b)\xi,\rho(a)\xi)|\\ 
& + & 
|(\Phi(y_{\alpha})\rho(b)\xi,\rho(a)\xi) - (\Phi(y)\rho(b)\xi,\rho(a)\xi)|\\
& + & 
|(\Phi(y)\rho(b)\xi,\rho(a)\xi) - (\Phi(y)\eta,\zeta)|\\
& = & 
|(\Phi(y_{\alpha})\eta,\zeta) - (\Phi(y_{\alpha})\rho(b)\xi,\rho(a)\xi)| 
+ 
|\phi(a^* \cdot y_{\alpha} \cdot b) - \phi(a^* \cdot y \cdot b)|\\
& + & 
|(\Phi(y)\rho(b)\xi,\rho(a)\xi) - (\Phi(y)\eta,\zeta)|\\
& \leq & 
|(\Phi(y_{\alpha})\eta,\zeta) - (\Phi(y_{\alpha})\rho(b)\xi,\zeta)|\\ 
& + & |(\Phi(y_{\alpha})\rho(b)\xi,\zeta) - (\Phi(y_{\alpha})\rho(b)\xi,\rho(a)\xi)|
+
|\phi(a^* \cdot y_{\alpha} \cdot b) - \phi(a^* \cdot y \cdot b)|\\
& + & 
|(\Phi(y)\rho(b)\xi,\rho(a)\xi) - (\Phi(y)\eta,\rho(a)\xi)| +  |(\Phi(y)\eta,\rho(a)\xi) - (\Phi(y)\eta,\zeta)|\\
& \leq &
\delta (\|\zeta\| + \|\eta\| + \|\rho(a)\xi\| + \|\rho(b)\xi\| + 1) 
\leq \delta (2\|\zeta\| + 2\|\eta\| + 2 \delta + 1).
\end{eqnarray*}
We thus showed that $\Phi(y_{\alpha})\to_{\alpha} \Phi(y)$ in the 
weak operator topology; since the net $(\Phi(y_{\alpha}))_{\alpha}$ is bounded, 
the convergence is in fact in the weak* topology.
It follows from Shmulyan's Theorem that the map $\Phi$ is weak* continuous.

Identity (\ref{eq_mod}) easily implies that 
there exists a (normal) map $\Psi : \cl S\to \cl B(K)$ such that 
$\Phi = \Psi^{(n)}$. 
Since $\Phi$ is hermitian and contractive, so is $\Psi$. 
By (\ref{eq_mod}) and Theorem \ref{l_sm}, the map $\Phi$, and hence $\Psi$, is completely contractive. 
Now (\ref{eq_mod}) implies
\begin{equation}\label{eq_modpsi}
\Psi(a \cdot z \cdot b) = \theta(a)\Psi(z)\theta(b), \ \ \ \ z\in \cl S, a,b\in \cl A.
\end{equation}
By (\ref{ope}), 
%\begin{equation}\label{eq_1}
$$1 = \phi(1) = (\Phi(1)\xi,\xi)\leq  \|\Phi(1)\|\|\xi\|^2 \leq 1.$$
%\end{equation}
Thus $\Phi(1)\xi = \xi$; by (\ref{eq_mod}),
$$\Phi(1)\rho(b)\xi = \rho(b)\Phi(1)\xi = \rho(b)\xi, \ \ \ \ b\in \cl B,$$
and since $\xi$ is cyclic for $\rho$, we conclude that $\Phi(1) = 1$.
It follows that $\Psi(1) = 1$. 

The map $\Psi$, constructed in the previous paragraph, 
depends on the element $x\in M_n(\cl S)$, and on the chosen $\epsilon$. 
Note that, by (\ref{eq_>ep}) and (\ref{ope}), $\left\|\Psi^{(n)}(x)\right\| > 1 - \epsilon$. 
Let $\gamma$ (resp. $\pi$) be the direct sum of the maps $\Psi$ 
(resp. $\theta$) as above, 
over all selfadjoint $x\in M_n(\cl S)$ with norm one, all $n\in \bb{N}$, and all $\epsilon \in (0,1)$. 
The map $\gamma$ is unital, weak* continuous, hermitian, and has the property that
if $x\in M_n(\cl S)$ is selfadjoint then $\|x\| = 1$ implies $\left\|\gamma^{(n)}(x)\right\| = 1$. 
This easily yields that $\gamma$ is completely positive and has a completely positive inverse.
As in the proof of \cite[Theorem 1.1]{bm}, 
the image of $\gamma$ is weak* closed and $\gamma$ is a weak* homeomorphism 
onto its range.
In addition, $\pi$ is a normal *-representation as a direct sum of such. 
Condition (\ref{eq_modc}) follows from (\ref{eq_modpsi}). 
\end{proof}

%%%%%%%%%%%%%%%%%%%%%%%%%%%%%%%%%%%%%%%%%%%%%%%%%%%%%%%%%%%%%%%%%%%
%%%%%%%%%%%%%%%%%%%%%%%%%%%%%%%%%%%%%%%%%%%%%%%%%%%%%%%%%%%%%%%%%%%%%%%%%%%%%%%%%%%%%%%%%%%%%%%%%%%%%%%%%%%%%%%%%%%%%%%%%%%%%%%%%%%%%%%%%%%%%%%%%%%%%%%%%%%%%%%%%%%%%%%%%%%%%%%%%%%%%%%%%%%%%%%%%%%%%%%%%

\section{The dual extremal operator $\cl A$-system structures}\label{s_deoass}

In this section, we study dual versions of the extremal operator $\cl A$-system structures
considered in Section \ref{s_eoss}. 
We start with the definition of a dual AOU space.
Note first that, if $(V,V^+,e)$ is an AOU space then the expression
$$\|v\| = \sup\{|f(v)| : f \mbox{ a state on } V\}$$
defines a norm on $V$, called the \emph{order norm} \cite{pt};
in the sequel we equip $V$ with its order norm. 
If $V$ is a dual Banach space, the weak* continuous functionals on $V$ will be called 
\emph{normal functionals}.

\begin{definition}\label{d_daous}
A \emph{dual AOU space} is an AOU space $(V,V^+,e)$, which is also a dual Banach space, and
\begin{itemize}
\item[(i)] the involution is weak* continuous;

\item[(ii)] $V^+$ is weak* closed, and 

\item[(iii)] for $v\in V$, $\|v\| = \sup\{|f(v)| : f \mbox{ a normal state on } V\}$, and 
the weak* topology of $V$ is determined by normal states of $V$. 
\end{itemize}
\end{definition}

Suppose that $(V,V^+,e)$ is a dual AOU space, and let $V_*$ be the predual of $V$. 
%When the cone $V^+$, the predual $V_*$ and the unit $e$ are 
%clear from the context, we simply talk about $V$ being a dual AOU space. 
Note that the algebraic tensor product $V_*\otimes M_n^*$ 
can be canonically embedded into the dual of $M_n(V)$.
By \emph{the weak* topology} on $M_n(V)$ we will mean the topology arising from this
duality; 
thus, $(x_{i,j}^{\alpha})\to_{\alpha} (x_{i,j})$ if and only if $x_{i,j}^{\alpha}\to_{\alpha} x_{i,j}$
for every $i,j$.

\begin{definition}
Let $\cl A$ be a W*-algebra. A dual AOU space $(V,V^+,e)$
will be called \emph{dual AOU $\cl A$-space} if 
\begin{itemize}
\item[(i)]  $(V,V^+,e)$ is an AOU $\cl A$-space, and

\item[(ii)] the left (and hence the right) $\cl A$-module 
action is separately weak* continuous.
\end{itemize}
\end{definition}

\begin{definition}\label{d_dopass}
Let $\cl A$ be a W*-algebra and $(V,V^+,e)$ be a dual AOU $\cl A$-space.
A matrix ordering $(C_n)_{n\in \bb{N}}$ on $V$ 
will be called a \emph{dual operator $\cl A$-system structure} on $V$
if $(V,(C_n)_{n\in \bb{N}},e)$ is a dual operator $\cl A$-system whose weak* topology coincides with that of $V$, and 
$C_1 = V^+$. 
\end{definition}

\begin{theorem}\label{th_weakscac}
Let $\cl A$ be a W*-algebra,
$(V,V^+,e)$ be a dual AOU $\cl A$-space and $(C_n)_{n\in \bb{N}}$ be 
an operator $\cl A$-system structure on $V$. 
The following are equivalent:

(i) \ $(C_n)_{n\in \bb{N}}$ is a dual operator $\cl A$-system structure on $V$;

(ii) $C_n$ is weak* closed for each $n\in \bb{N}$.
\end{theorem}
\begin{proof}
(i)$\Rightarrow$(ii) 
Let $\cl S = (V,(C_n)_{n\in \bb{N}},e)$. 
By Theorem \ref{th_repdoas}, 
there exist a Hilbert space $H$ and a 
complete order embedding $\gamma : \cl S\to \cl B(H)$ such that $\gamma(\cl S)$ is weak* closed and 
$\gamma$ is a weak* homeomorphism. 
Clearly, $M_n(\gamma(\cl S))^+$ is weak* closed in $M_n(\cl B(H))$. 
Note that the weak* topology on 
$M_n(\cl B(H)) = \cl B(H^n)$ is given by entry-wise weak* convergence. 
On the other hand, since $\gamma$ is a weak* homeomorphism, 
we have that if $((x_{i,j}^{\alpha}))_{\alpha}\subseteq M_n(V)$ and $(x_{i,j})\in M_n(V)$ 
then $(x_{i,j}^{\alpha})\to_{\alpha} (x_{i,j})$ weak* if and only if $\gamma(x_{i,j}^{\alpha})\to_{\alpha}\gamma(x_{i,j})$
for every $i,j$. It follows that $C_n$ is weak* closed.

(ii)$\Rightarrow$(i)
Let $\cl S = (V,(C_n)_{n\in \bb{N}},e)$. 
For each $n$, let 
$$\cl P_n = \left\{\phi : V \to M_n \ : \ \mbox{ weak* continuous unital completely positive map}\right\}.$$
Let $H = \oplus_{n\in \bb{N}}\oplus_{\phi\in \cl P_n} \bb{C}^n$ and let 
$J : V\to \cl B(H)$ be the map given by 
$J(x) = \oplus_{n\in \bb{N}}\oplus_{\phi\in \cl P_n} \phi(x)$. 
It is clear that $J$ is a weak* continuous completely positive map. 
In addition, by condition (iii) from Definition \ref{d_daous}, $J$ is isometric.

To show that $J$ is a complete order isomorphism, 
assume that $J^{(n)}(X)\geq 0$ for some $X = (x_{i,j})\in M_n(V)_h$
and that, by way of contradiction, $X$ does not belong to $C_n$.
The space $M_n(V)$, equipped with the topology of weak* convergence, 
is a locally convex topological vector space. 
By a geometric form of the Hahn-Banach Theorem,
there exists a functional $s : M_n(V)\to \bb{C}$, 
continuous with respect to the topology of entry-wise weak* convergence,
such that $s(C_n)\subseteq \bb{R}^+$ but $s(X) < 0$.
By \cite[Theorem 6.1]{Pa}, the map $\phi_s : V \to M_n$, given by 
$\phi_s(x) = (s_{i,j}(x))_{i,j}$ (and where $s_{i,j}(x) = s(E_{i,j}\otimes x)$), 
is completely positive. It is clear that $\phi_s$ is normal. 
In addition, $\phi_s^{(n)}$ does not map $X$ to a positive matrix. 
After normalisation, we may assume that $\phi_s$ is contractive.

Let $P = \phi_s(e)$; then $P$ is a positive contraction. 
Assume that ${\rm rank}(P) = k$ and let $Q$ be the projection onto $\ker(P)^{\perp}$.
It was shown in the proof of \cite[Theorem 13.1]{Pa} 
that, if $A\in M_{n,k}$ and $B\in M_{k,n}$ are matrices such that $A^* P A = I_k$ and  
$AB = Q$, and $\psi$ is the mapping given by $\psi(x) = A^* \phi_s(x) A$, then 
$\psi$ is a (unital completely positive) map such that $\psi^{(n)}(X)$ is not positive. 
Clearly, $\psi$ is normal, and hence an element of $\cl P_k$. This contradicts the fact that $J^{(n)}(X)\geq 0$.

To show that $J$ is a weak* homeomorphism, suppose that 
$J(x_{\alpha})\to_{\alpha} J(x)$ in the weak* topology, 
for some net $(x_{\alpha})\subseteq V$ and some element $x\in V$. 
Then $\phi(x_{\alpha})\to \phi(x)$ for all normal positive functionals $\phi$. 
By condition (iii) of Definition \ref{d_daous}, 
$x_{\alpha}\to x$ in the weak* topology of $V$. 

We finally note that $J(V)$ is weak* closed in $\cl B(H)$. Suppose that 
$J(x_{\alpha})\to T$, where $T\in \cl B(H)$ and 
$(x_{\alpha})_{\alpha}\subseteq V$ is a net such that the net $J(x_{\alpha})_{\alpha}$ is bounded.
Since $J$ is an isometry, $(x_{\alpha})_{\alpha}$ is also bounded, and hence has a  
subnet $(x_{\beta})_{\beta}$, weak* convergent to an element of $V$, say $x$. 
Since $J$ is weak* continuous, we conclude that 
$T = \lim_{\beta} J(x_{\beta}) = J(x)$, and hence $T\in J(V)$.
By the Krein-Smulyan, $J(V)$ is weak* closed. 

By the previous paragraphs, the weak* topology of $V$ 
coincides with the weak* topology of the operator system $\cl S$. 
It now follows that the $\cl A$-module operations on $\cl S$ are separately weak* continuous;
thus, $\cl S$ is a dual operator $\cl A$-system and the proof is complete. 
\end{proof}

As the next two statements show, if $(V,V^+,e)$ is a dual AOU $\cl A$-space 
then the minimal operator $\cl A$-system structure defined in Section \ref{s_eoss} is 
automatically a dual minimal operator $\cl A$-system structure.

\begin{theorem}\label{th_maxdual}
Let $\cl A$ be a W*-algebra and $(V,V^+,e)$ be a dual AOU $\cl A$-space.
Then $(C_n^{\min}(V;\cl A))_{n\in \bb{N}}$ is a dual operator $\cl A$-system structure.
\end{theorem}
\begin{proof}
Since the $\cl A$-module actions on $V$ are weak* continuous, $C_n^{\min}(V;\cl A)$ is weak* closed for each $n\in \bb{N}$.
By Theorem \ref{th_weakscac}, $(C_n^{\min}(V;\cl A))_{n\in \bb{N}}$ is a dual operator $\cl A$-system structure.
\end{proof}

\begin{theorem}\label{th_numin}
Let $\cl A$ be a W*-algebra and $(V,V^+,e)$ be a dual AOU $\cl A$-space. 

(i) Suppose that $\cl S$ is a dual operator $\cl A$-system 
and $\phi : \cl S\to V$ is a normal positive $\cl A$-bimodule map. 
Then $\phi$ is completely positive as a map from $\cl S$ into $\omin_{\cl A}(V)$. 

(ii) If $\cl T$ is a dual operator $\cl A$-system with underlying space $V$ and positive cone $V^+$, 
such that for every dual operator $\cl A$-system $\cl S$,
every normal positive $\cl A$-bimodule map $\phi : \cl S\to \cl T$ is completely positive, 
then there exists a unital normal $\cl A$-bimodule map $\psi : \cl T \to \omin_{\cl A}(V)$ 
that is a complete order isomorphism and a weak* homeomorphism.
\end{theorem}

\begin{proof}
(i) is a direct consequence of Theorem \ref{th_umin} (i). The proof of (ii) follows 
by a standards argument, similar to the one given in the proof of Theorem \ref{th_umin} (ii).
\end{proof}

In the remainder of the section, we consider the dual maximal operator $\cl A$-system structure.
For a W*-algebra $\cl A$ and a dual AOU $\cl A$-space $(V,V^+,e)$, 
set 
$$W_n^{\max}(V;\cl A) = \overline{C_n^{\max}(V;\cl A)}^{w^*}, \ \ \ \ n\in \bb{N}.$$

\begin{theorem}\label{th_wmaxa}
Let $\cl A$ be a W*-algebra and $(V,V^+,e)$ be a dual AOU $\cl A$-space.
Then $(W_n^{\max}(V;\cl A))_{n\in \bb{N}}$ is a dual operator $\cl A$-system structure on $V$. 
Moreover, if $(P_n)_{n\in \bb{N}}$ is a dual operator $\cl A$-system structure on $V$ then 
$W_n^{\max}(V;\cl A)\subseteq P_n$ for each $n\in \bb{N}$.
\end{theorem}
\begin{proof}
By Theorem \ref{th_cmaxa}, $(C_n^{\max}(V;\cl A))_{n\in \bb{N}}$ is an 
operator system $\cl A$-structure on $V$. It follows by the 
separate weak* continuity of the $\cl A$-module actions on $V$ 
and the definition of the $M_n(\cl A$)-module operations on $M_n(V)$ (see (\ref{eq_matrixmod})) 
that the family $(W_n^{\max}(V;\cl A))_{n\in \bb{N}}$ is $\cl A$-compatible. 

Since the element $e$ is a matrix order unit for $(D_n^{\max}(V;\cl A))_{n\in \bb{N}}$ 
(see Proposition \ref{minmax}) and $D_n^{\max}(V;\cl A)\subseteq W_n^{\max}(V;\cl A)$ for each $n\in \bb{N}$,
$e$ is a matrix order unit for $(W_n^{\max}(V;\cl A))_{n\in \bb{N}}$. 
To show that $e$ is an Archimedean matrix order unit for 
$(W_n^{\max}(V;\cl A))_{n\in \bb{N}}$, suppose that 
$X\in M_n(V)$ is such that $X + re_n\in W_n^{\max}(V;\cl A)$ for all $r > 0$. 
Since $X + re_n\to_{r\to 0} X$ in the weak* topology and $W_n^{\max}(V;\cl A)$ is weak* closed, 
$X\in W_n^{\max}(V;\cl A)$. 

It follows that $(V, (W_n^{\max}(V;\cl A))_{n\in \bb{N}}, e)$ is an operator $\cl A$-system; 
by condition (ii) of Definition \ref{d_daous},
$V^+ = W_1^{\max}(V;\cl A)$.
Since its cones are weak* closed, Theorem \ref{th_weakscac} implies that 
it is a dual operator $\cl A$-system. 

Suppose that $(P_n)_{n\in \bb{N}}$ is a dual operator $\cl A$-system structure on $V$. 
Fix $n\in \bb{N}$. By Theorem \ref{th_cmaxa}, $C_n^{\max}(V;\cl A)\subseteq P_n$. 
By Theorem \ref{th_weakscac}, $P_n$ is weak* closed. 
It follows that $W_n^{\max}(V;\cl A)\subseteq P_n$. 
\end{proof}

We denote by $\omax_{\cl A}^{w^*}(V)$ the operator system $(V,(W_n^{\max}(V;\cl A))_{n\in \bb{N}},e)$.

\begin{theorem}\label{th_nacp}
Let $\cl A$ be a W*-algebra and $(V,V^+,e)$ be a dual AOU $\cl A$-space. 

(i) Suppose that $\cl S$ is a dual operator $\cl A$-system 
and $\phi : V\to \cl S$ is a normal positive $\cl A$-bimodule map. 
Then $\phi$ is completely positive as a map from $\omax_{\cl A}^{w^*}(V)$ into $\cl S$. 

(ii) If $\cl T$ is a dual operator $\cl A$-system with underlying space $V$ and positive cone $V^+$, 
such that for every dual operator $\cl A$-system $\cl S$,
every normal positive $\cl A$-bimodule map $\phi : \cl T\to \cl S$ is completely positive, 
then there exists a unital normal $\cl A$-bimodule map $\psi : \cl T\to \omax^{w^*}_{\cl A}(V)$ 
that is a complete order isomorphism and a weak* homeomorphism.
\end{theorem}
\begin{proof}
(i) 
By Theorem \ref{th_acp} (i), $\phi^{(n)}(C_n^{\max}(V;\cl A))\subseteq M_n(\cl S)^+$. 
Since $\phi$ is weak* continuous and $M_n(\cl S)^+$ is weak* closed, 
$\phi^{(n)}(W_n^{\max}(V;\cl A))\subseteq M_n(\cl S)^+$.

(ii) similar to the proof of Theorem \ref{th_umin} (ii). 
\end{proof}

\noindent {\bf Remark. } 
Let $\cl A$ be a W*-algebra and 
$\frak{A}^{w^*}_{\cl A}$ (resp. $\frak{S}^{w^*}_{\cl A}$) be the category, whose objects are dual AOU $\cl A$-spaces
(resp. dual operator $\cl A$-systems) and whose morphisms are weak* continuous unital positive
(resp. weak* continuous unital completely positive) maps. 
It is easy to see that the 
correspondences $V\to \omin^{w^*}_{\cl A}(V)$ and $V\to \omax^{w^*}_{\cl A}(V)$
are covariant functors from $\frak{A}^{w^*}_{\cl A}$ into $\frak{S}^{w^*}_{\cl A}$, here $\omin^{w^*}_{\cl A}(V)= \omin_{\cl A}(V)$ 
as per Theorem \ref{th_maxdual}.

%%%%%%%%%%%%%%%%%%%%%%%%%%%%%%%%%%%%%%%%%%%%%%%%%%%%%%%%%%%%%%%%%%%
%%%%%%%%%%%%%%%%%%%%%%%%%%%%%%%%%%%%%%%%%%%%%%%%%%%%%%%%%%%%%%%%%%%%%%%%%%%%%%%%%%%%%%%%%%%%%%%%%%%%%%%%%%%%%%%%%%%%%%%%%%%%%%%%%%%%%%%%%%%%%%%%%%%%%%%%%%%%%%%%%%%%%%%%%%%%%%%%%%%%%%%%%%%%%%%%%%%%%%%%%

\section{Inflated Schur multipliers}\label{s_ism}

In this section, we introduce an operator-valued version of classical measurable 
Schur multipliers, and characterise them in a fashion, similar to 
the well-known descriptions in the scalar-valued case \cite{Gro, peller}.

Let $(X,\mu)$ be a standard measure space. 
We denote by $\chi_{\alpha}$ the characteristic function of a measurable set $\alpha\subseteq X$.
If $f$ and $g$ are measurable functions
defined on $X$, we write $f\sim g$ when $f(x) = g(x)$ for almost all $x \in X$. 
Throughout the section, let $H = L^2(X,\mu)$ and 
fix a separable Hilbert space $K$.
For a function $a\in L^{\infty}(X,\mu)$, let 
$M_a$ be the operator on $H$ given by 
$M_a f = af$, $f\in H$, and set
$$\cl D = \left\{M_a : a\in L^{\infty}(X,\mu)\right\}.$$ 
We denote by $H\otimes K$ the Hilbertian tensor product of 
$H$ and $K$. Note that $H\otimes K$ is unitarily equivalent 
to the space $L^2(X,K)$ of all weakly measurable functions 
$g : X\to K$ such that $\|g\|_2 := \left(\int_X \|g(x)\|^2d\mu(x)\right)^{1/2} < \infty$. 

If $\cl U\subseteq \cl B(H)$ and $\cl V\subseteq \cl B(K)$, 
we denote by $\cl U\bar{\otimes}\cl V$ the spacial weak* tensor product 
of $\cl U$ and $\cl V$. 
We write $\cl M(X,\cl B(K))$ for the space of all functions $F : X \to \cl B(K)$ such that, 
for all $\xi_0\in K$, the functions $x\to F(x)\xi_0$ and $x\to F(x)^*\xi_0$ are weakly measurable.
Note that
$\cl D\bar\otimes\cl B(K)$ can be canonically identified 
with the space $L^{\infty}(X,\cl B(K))$ of all bounded functions $F$ in $\cl M(X,\cl B(K))$ \cite{t}.
Through this identification, a function $F$ gives rise to the operator 
$M_F\in \cl B(L^2(X,K))$, defined by 
$$(M_F\xi)(x) = F(x)(\xi(x)), \ \ \ x\in X, \ \xi\in L^2(X,K).$$ 
It is easy to see that if $k\in \cl M(X\times X,\cl B(K))$ then the function 
$(x,y)\to \|k(x,y)\|$ is measurable as a function from $X\times X$ into $[0,+\infty]$. 
Let $L^2(X\times X,\cl B(K))$ be the space of all functions $k\in \cl M(X\times X,\cl B(K))$ for which 
$$\|k\|_2:= \left(\int_{X\times X}\|k(x,y)\|^2 d\mu(x) d\mu(y) \right)^{1/2} < \infty.$$
(Note that the functions from the space $L^2(X\times X,\cl B(K))$ need not be weakly measurable.)
If $k\in L^2(X\times X,\cl B(K))$ and $\xi,\eta\in L^2(X,K)$ then,
by \cite[Lemma 7.5]{t}, the function $(x,y)\to \left(k(x,y)(\xi(y)),\eta(x)\right)$ is measurable. 
Standard arguments (see \cite[p. 391]{mtt}) show that 
the formula
$$(T_k\xi,\eta) = \int_{X\times X} \left(k(x,y)(\xi(y)),\eta(x)\right) d\mu(y)d\mu(x), \ x,y\in X, \xi,\eta\in L^2(X,K),$$
defines a bounded operator on $L^2(X,K)$ with 
$\|T_k\|\leq \|k\|_2$.
If $K = \bb{C}$, the operators of the form $T_k$ are precisely the Hilbert-Schmidt operators on $H$.

\begin{remark}\label{r_zero}
For an element $k\in L^2(X\times X,\cl B(K))$, we have that $T_k = 0$ if and only if $k(x,y) = 0$ for almost all $(x,y)\in X\times X$.
\end{remark}
\begin{proof}
Suppose that $T_k = 0$; then, for $\xi,\eta\in K$ and $f,g\in L^2(X)$, we have 
$\int_{X\times X} f(x) g(y) (k(x,y)\xi,\eta) d\mu(y)d\mu(x) = 0$. 
Thus, $(k(x,y)\xi,\eta) = 0$ almost everywhere. Since $K$ is separable and $k(x,y)$ is bounded for all $x,y\in X$, 
this implies that $k(x,y) = 0$ almost everywhere. The converse direction is trivial.
\end{proof}

We equip the linear space $\{T_k : k\in L^2(X\times X,\cl B(K))\}$
with the operator space structure arising from its inclusion into $\cl B(H\otimes K)$. 
Similarly, whenever $\cl S$ is an operator system and $\cl S_0\subseteq \cl S$ is a self-adjoint 
(not necessarily unital) subspace of $\cl S$, we equip $\cl S_0$ with the matrix ordering
inherited from $\cl S$, and thus talk about a linear map from $\cl S_0$ into an operator system $\cl T$
being positive or completely positive. 

For functions $\nph\in L^{\infty}(X\times X,\cl B(K))$ and $k\in L^2(X\times X)$, let 
$\nph k : X\times X\to \cl B(K)$ be the function given by 
$$(\nph k)(x,y) = k(x,y)\nph(x,y), \ \ \ x,y\in X.$$
It is straightforward to check that $\nph k\in L^2(X\times X,\cl B(K))$.

\begin{definition}\label{def_opv}
A function $\nph\in L^{\infty}(X\times X,\cl B(K))$ 
will be called 
an \emph{(inflated) Schur multiplier} if the map
$$T_k \longrightarrow T_{\nph k}, \ \ k\in L^2(X\times X),$$
is completely bounded.
\end{definition}

We will denote by $\frak{S}(X,K)$ the space of all inflated Schur multipliers with values in $\cl B(K)$.
If $\nph\in \frak{S}(X,K)$ then the map 
$S_{\nph} : T_k\to T_{\nph k}$ defined on the space $\cl S_2(H)$ of all Hilbert-Schmidt operators on $H$ 
extends to a completely bounded map 
from $\cl K(H)$ into $\cl B(H\otimes K)$, which will be denoted in the same way. 
By taking the second dual of $S_{\nph}$, and composing with the weak* continuous 
projection from $\cl B(H\otimes K)^{**}$ onto $\cl B(H\otimes K)$, 
we obtain a completely bounded weak* continuous map from $\cl B(H)$ into $\cl B(H\otimes K)$
which for simplicity will still be denoted by $S_{\nph}$.

\begin{theorem}\label{th_deou}
Let $\nph\in L^{\infty}(X\times X,\cl B(K))$. The following are equivalent:

(i) \ $\nph\in \frak{S}(X,K)$;

(ii) there exist functions $A_i \in L^{\infty}(X,\cl B(K))$ and $B_i \in L^{\infty}(X,\cl B(K))$, $i\in \bb{N}$, such that 
the series $\sum_{i=1}^{\infty} A_i(x)A_i(x)^*$ and $\sum_{i=1}^{\infty} B_i(y)^*B_i(y)$
converge almost everywhere in the weak* topology,
$$\esssup_{x\in X} \left\|\sum_{i=1}^{\infty} A_i(x)A_i(x)^*\right\| < \infty, \ \ 
\esssup_{y\in X} \left\|\sum_{i=1}^{\infty} B_i(y)^*B_i(y)\right\| < \infty,$$
\noindent and
\begin{equation}\label{eq_aibi}
\nph(x,y) = \sum_{i=1}^{\infty} A_i(x)B_i(y), \ \ \ \mbox{ a.e. on } X\times X,
\end{equation}
where the sum is understood in the weak* topology. 
\end{theorem}

\begin{proof} 
(ii)$\Rightarrow$(i)
%By the assumptions, we have that the partial sums of the series in (\ref{eq_aibi}) 
%are uniformly bounded; in particular, that $\nph\in L^{\infty}(X\times X,\cl B(K))$. 
Considering $A_i, B_i\in \cl D\bar\otimes\cl B(K)$, $i\in \bb{N}$, 
the assumptions imply that $A = (A_i)_{i\in \bb{N}}$ (resp. $B = (B_i)_{i\in \bb{N}}$) is a bounded 
row (resp. column) operator.
It follows that the map $\Psi : \cl B(H)\to \cl B(H\otimes K)$, given by
$$\Psi(T) = \sum_{i=1}^{\infty} A_i(T\otimes I) B_i, \ \ \ T\in \cl B(H),$$
is well-defined and completely bounded. 
Let $k\in L^2(X\times X) \cap L^{\infty}(X\times X)$, $\xi,\eta\in K$ and $f,g\in L^2(X)\cap L^1(X)$.
For almost all $(x,y)\in X\times X$, we have 
\begin{eqnarray*}
& & 
\left|k(x,y)f(y)\overline{g(x)}\left(\nph(x,y)\xi,\eta\right)\right| \\
& \leq & 
\|k\|_{\infty} |f(y)| |g(x)|\sum_{i=1}^{\infty} \left|(B_i(y)\xi,A_i(x)^*\eta)\right| \\
& \leq & 
\|k\|_{\infty} |f(y)| |g(x)|\sum_{i=1}^{\infty} \|B_i(y)\xi\| \|A_i(x)^*\eta\|\\
& \leq & 
\|k\|_{\infty} |f(y)| |g(x)| \left(\sum_{i=1}^{\infty}  \|B_i(y)\xi\|^2\right)^{1/2}  
\left(\sum_{i=1}^{\infty} \|A_i(x)^*\eta\|^2\right)^{1/2} \\
& \leq & 
\|k\|_{\infty} |f(y)|  |g(x)| \|A\| \|B\| \|\xi\| \|\eta\|,
\end{eqnarray*}
while the function $(x,y)\to  |f(y)|  |g(x)|$ is integrable with respect to $\mu\times\mu$. 
By the Lebesgue Dominated Convergence Theorem, we now have 
\begin{eqnarray*}
& & 
(\Psi(T_k)(f\otimes\xi),g\otimes\eta)\\
& = & 
\left(\sum_{i=1}^{\infty} A_i(T_k\otimes I) B_i (f\otimes \xi),g\otimes \eta\right)\\
& = & 
\sum_{i=1}^{\infty} \int_{X\times X} k(x,y)f(y)\overline{g(x)}(B_i(y)\xi,A_i(x)^*\eta)d\mu(x)d\mu(y)\\
& = & 
\int_{X\times X} k(x,y)f(y)\overline{g(x)}\left(\left(\sum_{i=1}^{\infty} A_i(x)B_i(y)\right)\xi,\eta\right)d\mu(x)d\mu(y)\\
& = & 
\int_{X\times X} k(x,y)f(y)\overline{g(x)}\left(\nph(x,y)\xi,\eta\right)d\mu(x)d\mu(y)\\
& = & 
\int_{X\times X} f(y)\overline{g(x)}\left((\nph k)(x,y)\xi,\eta\right)d\mu(x)d\mu(y)\\
& = & 
\left(T_{\nph k}(f\otimes\xi),g\otimes \eta\right).
\end{eqnarray*}
By linearity and the density of $L^2(X\times X) \cap L^{\infty}(X\times X)$ in $L^2(X\times X)$
and of $L^2(X) \cap L^1(X)$ in $L^2(X)$, 
it follows that $\nph\in \frak{S}(X,K)$ and $\Psi = S_{\nph}$.

(i)$\Rightarrow$(ii)
Let $\nph\in \frak{S}(X,K)$. 
For $k\in L^2(X\times X)$,  $a,b\in L^{\infty}(X)$, $\xi,\eta\in K$ and $f,g\in L^2(X)$, we have 
\begin{eqnarray*}
& & 
\left(S_{\nph}(M_bT_kM_a)(f\otimes\xi),g\otimes\eta\right)\\
& = & 
\int_{X\times X} a(y) b(x) f(y)\overline{g(x)}\left((\nph k)(x,y)\xi,\eta\right)d\mu(x)d\mu(y)\\
& = & 
\left((M_b\otimes I) S_{\nph}(T_k)(M_a\otimes I)(f\otimes\xi),g\otimes\eta\right).
\end{eqnarray*}
By continuity, 
$$S_{\nph}(BTA) = (B\otimes I)S_{\nph}(T)(A\otimes I), \ \ \ T\in \cl K(H), A,B\in \cl D.$$
Let $\Phi_1 : \cl K(H) \otimes 1 \to \cl B(H\otimes K)$ be the map given by $\Phi_1(T\otimes I) = S_{\nph}(T)$;
then $\Phi_1$ is a completely bounded $\cl D\otimes 1$-bimodule map. 
Using \cite[Exercise 8.6 (ii)]{Pa}, we can find a 
completely bounded weak* continuous $\cl D\otimes 1$-bimodule map 
$\Phi_2 : \cl B(H\otimes K) \to \cl B(H\otimes K)$ extending $\Phi_1$.
%Let $\Phi_3 : \cl K(H\otimes K) \to \cl B(H\otimes K)$ be the restriction of $\Phi_2$, and 
%$\Phi_4 : \cl B(H\otimes K) \to \cl B(H\otimes K)$ be the canonical weak* continuous extension 
%of $\Phi_3$. Note that $\Phi_4$ is a $\cl D\otimes 1$-bimodule map. 
By \cite{haag}, there exist a bounded row operator $A = (A_i)_{i=1}^{\infty}$ and 
a bounded column operator $B = (B_i)_{i\in \bb{N}}$, 
where $A_i, B_i\in \cl D\bar\otimes\cl B(K)$, $i\in \bb{N}$, such that 
$$\Phi_2(T) = \sum_{i=1}^{\infty} A_i T B_i, \ \ \ T\in \cl B(H\otimes K).$$
Using the identification $\cl D\bar\otimes\cl B(K) \equiv L^{\infty}(X,\cl B(K))$, we 
consider $A_i$ (resp. $B_i$) as a function $A_i : X\to \cl B(K)$
(resp. $B_i : X\to \cl B(K)$). 
The boundedness of $A$ and $B$ now imply that 
there exists a null set $N\subseteq X$ such that the series 
$$\sum_{i=1}^{\infty} A_i(x)A_i(x)^* \ \ \mbox{ and } \ \ \sum_{i=1}^{\infty} B_i(y)^* B_i(y)$$
are weak* convergent whenever $x,y\not\in N$.
If $(x,y)\not\in N\times N$ then the series $\sum_{i=1}^{\infty} A_i(x) B_i(y)$ is weak* convergent. 
As in the first part of the proof, we conclude that 
$\nph(x,y)$ coincides with its sum for almost all $(x,y)$. 
\end{proof}

An inspection of the proof of Theorem \ref{th_deou} shows the following 
description of inflated Schur multipliers.

\begin{remark}\label{r_modc0}
The following are equivalent, for a completely bounded map $\Phi : \cl K(H)\to \cl B(H\otimes K)$:

(i) \ $\Phi(BTA) = (B\otimes I)\Phi(T)(A\otimes I)$, for all $T\in \cl K(H)$ and all $A,B\in \cl D$;

(ii) there exists a Schur multiplier $\nph \in \frak{S}(X,K)$ such that $\Phi = S_{\nph}$.
\end{remark}

\begin{definition}\label{def_opvp}
A Schur multiplier $\nph\in \frak{S}(X,K)$ will be called 
\emph{positive} if the map $S_{\nph} : \cl B(H)\to \cl B(H\otimes K)$ is positive.
\end{definition}

For the next theorem, note that, if $\nph\in L^{\infty}(X\times X,\cl B(K))$ and 
$\alpha\subseteq X$ is a subset of finite measure
then the function $\nph \chi_{\alpha\times\alpha}$ belongs to $L^2(X\times X,\cl B(K))$
and hence the operator 
$T_{\nph \chi_{\alpha\times\alpha}} : H\to H\otimes K$ is well-defined.

\begin{theorem}\label{th_modc}
The following are equivalent, for a Schur multiplier $\nph\in \frak{S}(X,K)$:

(i) \ \ $\nph$ is positive;

(ii) \ the map $S_{\nph} : \cl B(H)\to \cl B(H\otimes K)$ is completely positive;

(iii) for every subset $\alpha\subseteq X$ of finite measure, 
the operator $T_{\nph \chi_{\alpha\times\alpha}}$ is positive;

(iv) \ there exist functions $A_i \in L^{\infty}(X,\cl B(K))$, $i\in \bb{N}$, such that 
the series $\sum_{i=1}^{\infty} A_i(x)A_i(x)^*$ converges almost everywhere in the weak* topology, 
$$\esssup_{x\in X} \left\|\sum_{i=1}^{\infty} A_i(x)A_i(x)^*\right\| < \infty,$$
and
$$\nph(x,y) = \sum_{i=1}^{\infty} A_i(x)A_i(y)^*, \ \ \ \mbox{ a.e. on } X\times X.$$
\end{theorem}
\begin{proof}
(i)$\Rightarrow$(iii) 
Let $\alpha\subseteq X$ be a subset of finite measure. 
Then $\chi_{\alpha}\in H$; let $\chi_{\alpha}\otimes\chi_{\alpha}^*$ be the corresponding 
(positive) rank one operator. Then
$$ T_{\nph \chi_{\alpha\times\alpha}} = S_{\nph}(\chi_{\alpha}\otimes\chi_{\alpha}^*),$$
and the conclusion follows. 

(iii)$\Rightarrow$(ii) 
Let $n\in \bb{N}$, $X_i = X$ for $i = 1,\dots,n$, $Y = X_1\cup\dots\cup X_n$
and $\nu$ be the disjoint sum of $n$ copies of the measure $\mu$. 
Identify $\bb{C}^n\otimes H$ with $L^2(Y,\nu)$, 
and define $\psi : Y\times Y\to \cl B(K)$ by letting 
$\psi(x,y) = \nph(x,y)$ if $(x,y) \in X_i\times X_j = X\times X$.
Note that $S_{\psi} = \id_{M_n}\otimes S_{\nph}$ and hence $\psi \in \frak{S}(Y,K)$. 
Let $\alpha\subseteq X$ have finite measure and 
$J\in M_n$ be the matrix all of whose entries are equal to $1$.
Let $\alpha_i \subseteq X_i$ be the set that coincides with $\alpha$, $i = 1,\dots,n$,
and $\tilde{\alpha} = \cup_{i=1}^n \alpha_i$; 
we have that 
\begin{equation}\label{eq_J}
T_{\psi \chi_{\tilde{\alpha} \times \tilde{\alpha}}} \equiv J \otimes T_{\nph \chi_{\alpha\times\alpha}}.
\end{equation}
By assumption, $T_{\nph \chi_{\alpha\times\alpha}}$ is positive; thus, by (\ref{eq_J}),
$T_{\psi \chi_{\tilde{\alpha}\times\tilde{\alpha}}}$ is positive. 
For $g\in L^{\infty}(Y,\nu)\cap L^2(Y,\nu)$ and $h\in L^{\infty}(\tilde{\alpha})$, we have
$$\left(S_{\psi}(g \otimes g^*)h,h\right) = 
\left(T_{\psi \chi_{\tilde{\alpha} \times \tilde{\alpha}}}(gh),gh\right)\geq 0.$$
Since the set
$$\left\{h\in L^2(Y,\nu) : \exists \mbox{ a set of finite measure } \alpha \subseteq X 
\mbox{ with } h\in L^{\infty}(\tilde{\alpha})\right\}$$
is dense in $L^2(Y,\nu)$, we have that 
$S_{\psi}(g \otimes g^*) \in \cl B(H\otimes K)^+$. 
By weak* continuity,
$S_{\psi}(T) \in \cl B(H\otimes K)^+$ whenever $T\in \cl B(L^2(Y,\nu))^+$. 
Thus, $S_{\psi}$ is positive, that is, $S_{\nph}$ is $n$-positive. 

(ii)$\Rightarrow$(i) is trivial. 

(ii)$\Rightarrow$(iv) follows from the proof of Theorem \ref{th_deou} by noting that 
in the case $S_{\nph}$ is completely positive, one can choose $B_i = A_i^*$, $i\in \bb{N}$. 

(iv)$\Rightarrow$(i) follows from the proof of Theorem \ref{th_deou}. 
\end{proof}

%%%%%%%%%%%%%%%%%%%%%%%%%%%%%%%%%%%%%%%%%%%%%%%%%%
%%%%%%%%%%%.    POSITIVE EXTENSIONS     %%%%%%%%%%%%%%%%%%%%%%%
%%%%%%%%%%%%%%%%%%%%%%%%%%%%%%%%%%%%%%%%%%%%%%%%%%

\section{Positive extensions}\label{s_pe}

In this section, we 
apply our results on maximal operator system $\cl A$-structures 
to questions about positive extensions of inflated Schur multipliers. 
We first recall some measure theoretic background from \cite{a} and \cite{eks}, required in the sequel. 
A subset
$E\subseteq X\times X$ is called \emph{marginally null} if $E\subseteq
(M\times X)\cup (X\times M)$, where $M\subseteq X$ is null.  We call
two subsets $E,F\subseteq X\times X$ {\it marginally equivalent} 
(resp. {\it  equivalent}),
and write $E\cong F$ (resp. $E\sim F$), 
if their symmetric difference is marginally null (resp. null with respect to 
product measure).
We say that $E$ is \emph{marginally contained} in $F$ (and write
$E\subseteq_{\omega} F$) if the set difference $E\setminus F$ is
marginally null.  A measurable subset $\kappa\subseteq X\times X$ is called
\begin{itemize}
\item a \emph{rectangle} if $\kappa = \alpha\times\beta$ where
$\alpha,\beta$ are measurable subsets of~$X$;
\item {\it $\omega$-open} if it is marginally equivalent to a countable union of rectangles, and 
\item {\it $\omega$-closed} if its complement $\kappa^c$ is $\omega$-open.
\end{itemize}
Recall that, by~\cite{stt_cl}, if $\cl E$ is any collection of 
$\omega$-open sets then there exists a smallest, up to marginal
equivalence, $\omega$-open set $\cup_{\omega}\cl E$, called the
\emph{$\omega$-union} of $\cl E$, such that every
set in~$\cl E$ is marginally contained in $\cup_{\omega}\cl E$.
Given a measurable set $\kappa$, one defines its
\emph{$\omega$-interior} to be
\[\ointer(\kappa) =
\bigcup\mbox{}_{\omega}\left\{R : R \, \mbox{ is a rectangle with } R \subseteq_{\omega} \kappa\right\}.\]
The \emph{$\omega$-closure} $\ocl(\kappa)$ of~$\kappa$ is defined to be 
the complement of $\ointer(\kappa^c)$. 
For a set $\kappa\subseteq X\times X$, we write $\hat{\kappa} =
\{(x,y)\in X\times X : (y,x)\in \kappa\}$.
The subset~$\kappa\subseteq X\times X$ is said to be 
\emph{generated by rectangles}  
if $\kappa\cong\ocl(\ointer(\kappa))$  \cite{eks, llt}.

For any $\omega$-closed subset 
$\kappa\subseteq X\times X$, let 
$$\cl S_2(\kappa) = \left\{T_k :  k \in L^2(\kappa)\right\}, \ \ \cl S_0(\kappa) = \overline{\cl S_2(\kappa)}^{\|\cdot\|} \ \mbox{ and } \
\cl S(\kappa) = \overline{\cl S_2(\kappa)}^{w^*},$$
where $L^2(\kappa)$ is the space of functions in $L^2(X\times X)$ which are supported
on $\kappa$, up to a set of zero product measure. 
Note that the spaces $\cl S_2(\kappa)$, $\cl S_0(\kappa)$ and $\cl S(\kappa)$ are 
$\cl D$-bimodules. 
We equip them with the operator space structures inherited from $\cl B(H)$.

Partially defined scalar-valued Schur multipliers were defined in \cite{llt}. Here we extend this 
notion to the operator-valued setting. 

\begin{definition}\label{d_newschur}
  Let $\kappa\subseteq X\times X$ be a subset generated by rectangles.
  A function $\nph \in L^{\infty}(\kappa,\cl B(K))$ will be called a 
  \emph{partially defined Schur multiplier} if the map $S_{\nph}$ from 
  $\cl S_2(\kappa)$ into $\cl B(H\otimes K)$, given by 
  $$S_{\nph}(T_{k}) = T_{\nph k}, \ \ \ k\in L^2(\kappa),$$
  is completely bounded.
\end{definition}

\begin{remark}\label{r_eqli}
For Schur multipliers $\nph, \psi \in L^{\infty}(\kappa,\cl B(K))$, we have that $S_{\nph} = S_{\psi}$ 
if and only if $\nph\sim\psi$.
\end{remark}

\begin{proof}
Suppose $\nph, \psi \in L^{\infty}(\kappa,\cl B(K))$ are such that $S_{\nph} = S_{\psi}$.
Then $T_{\nph k} = T_{\psi k}$ for every $k\in L^2(\kappa)$. 
By Remark \ref{r_zero}, $\nph k \sim \psi k$. It now easily follows that $\nph\sim\psi$. 
The converse implication follows by reversing the previous steps. 
\end{proof}

Let $\kappa\subseteq X\times X$ be a subset generated by rectangles.
We note that the map $S_{\nph}$ from Definition \ref{d_newschur} is $\cl D$-bimodular. 
In addition, if $\psi \in \frak{S}(X,K)$ is given as in Definition \ref{def_opv}, then its restriction $\psi|_{\kappa} : \kappa\to \cl B(K)$
is an inflated Schur multiplier.

\begin{proposition}\label{p_chsch} 
Let $K$ be a separable Hilbert space, $\kappa\subseteq X\times X$ a subset 
generated by rectangles and 
$\nph \in L^{\infty}(\kappa,\cl B(K))$. The following are equivalent:

(i) \ \  $\nph$ is a Schur multiplier;

(ii) \ there exists a Schur multiplier $\psi : X\times X\to \cl B(K)$
  such that $\psi|_{\kappa} \sim \nph$;

(iii) there exists a unique completely bounded map 
$\Phi_0 : \cl S_0(\kappa) \to \cl B(H\otimes K)$ such that 
$\Phi_0(T_k) = T_{\nph k}$, for each $k\in L^2(\kappa)$;

(iv) \ there exists a unique completely bounded weak* continuous map
$\Phi : \cl S(\kappa) \to \cl B(H\otimes K)$ such that $\Phi(T_k) = T_{\nph k}$, for each $k\in L^2(\kappa)$.
\end{proposition}

\begin{proof} 
(i)$\Rightarrow$(ii)
Since $\nph$ is a Schur multiplier, the map 
$\Phi_2 : \cl S_2(\kappa)\to \cl B(H\otimes K)$, given by 
$\Phi_2(T_k) = T_{\nph k}$, extends to a completely bounded linear map
$\Phi_0 : \cl S_0(\kappa)\to \cl B(H\otimes K)$. 
By continuity,
$$\Phi_0(BTA) = (B\otimes I)\Phi_0(T)(A\otimes I), \ \ \ T\in \cl S_0(\kappa), A,B \in \cl D.$$
Let 
$\hat{\Phi} : \cl S_0(\kappa)\otimes 1 \to \cl B(H\otimes K)$ be the map given by 
$$\hat{\Phi}(T\otimes I) = \Phi_0(T), \ \ \  T\in \cl S_0(\kappa).$$
By~\cite[Exercise 8.6 (ii)]{Pa}, there exists a completely bounded 
$\cl D\otimes 1$-bimodule map 
$\hat{\Phi}_1 : \cl B(H\otimes K) \to \cl B(H\otimes K)$, extending $\hat{\Phi}$.
Let $\hat{\Psi} : \cl K(H)\otimes 1\to \cl B(H\otimes K)$ be the restriction of $\hat{\Phi}_1$; then
$\hat{\Psi}|_{\cl S_0(\kappa)\otimes 1} = \hat{\Phi}$. Let 
$\Psi : \cl K(H)\to \cl B(H\otimes K)$ be given by $\Psi(T) = \hat{\Psi}(T\otimes I)$.
Clearly,
$$\Psi(BTA) = (B\otimes I)\Psi(T)(A\otimes I), \ \ \ T\in \cl K(H), A,B \in \cl D.$$
By Remark~\ref{r_modc0}, there exists 
$\psi\in \frak{S}(X,K)$ such that $\Psi = S_{\psi}$. 
For every $k\in L^2(\kappa)$ we have 
$S_{\psi}(T_k) = S_{\nph}(T_k)$.
By Remark \ref{r_eqli}, $\psi|_{\kappa}\sim \nph$.

(ii)$\Rightarrow$(iv) Take $\Phi = S_{\psi}|_{\cl S(\kappa)}$. 
The uniqueness of $\Phi$ follows from the fact that 
the Hilbert-Schmidt operators with integral kernels in $L^2(\kappa)$ 
are weak* dense in $\cl S(\kappa)$.

(iv)$\Rightarrow$(iii)$\Rightarrow$(i) are trivial. 
\end{proof}

If $\nph : \kappa \to \cl B(K)$ is a Schur multiplier then 
we will denote still by $S_{\nph}$ the weak* continuous map
defined on $\cl S(\kappa)$ whose existence was established in Proposition \ref{p_chsch} (iv).

We say that a subset $\kappa\subseteq X\times X$ is \emph{symmetric} if~$\kappa\cong\hat\kappa$.
We call $\kappa$ a \emph{positivity domain} \cite{llt} if $\kappa$ is symmetric,
generated by rectangles and the diagonal 
$\Delta := \{(x,x) : x\in X\}$ is marginally contained in $\kappa$.
The following was established in \cite{llt}:

\begin{proposition}\label{p_opsc}
If $\kappa\subseteq X\times X$ is generated by rectangles, then the following are equivalent:

(i) \ $\cl S(\kappa)$ is an operator system;

(ii) $\kappa$ is a positivity domain.
\end{proposition}

Let $\nph : \kappa\to \cl B(K)$ be a Schur multiplier. We say that the Schur multiplier $\psi : X\times X\to \cl B(K)$ 
is a \emph{positive extension} of $\nph$ if $\psi$ is positive and $\psi|_{\kappa} \sim \nph$.

\begin{proposition}\label{p_cpext}
Let $\kappa$ be a positivity domain and $\nph : \kappa\to \cl B(K)$ be a Schur multiplier. 
The following are equivalent:

(i) \ $\nph$ has a positive extension;

(ii) the map $S_{\nph} : \cl S(\kappa) \to \cl B(H\otimes K)$ is completely positive.
\end{proposition}
\begin{proof} 
(i)$\Rightarrow$(ii) Suppose that $\psi : X\times X\to \cl B(K)$ is a positive extension of $\nph$. 
By Theorem \ref{th_modc}, $S_{\psi}$ is completely positive. 
On the other hand, $S_{\psi}|_{\cl S(\kappa)} = S_{\psi|_{\kappa}}$. 
Since $\psi|_{\kappa} = \nph$, we conclude that $S_{\nph}$ is completely positive. 

(ii)$\Rightarrow$(i) 
Let $\Phi_0$ be the restriction of $S_{\nph}$ to 
$\cl S_0(\kappa) + \bb{C}I$; clearly, $\Phi_0$ is a completely positive map. 
By Arveson's Extension Theorem, 
there exists a completely positive map $\Psi_0 : \cl K(H) + \bb{C}I\to \cl B(H\otimes K)$ extending $\Phi_0$. 
The restriction $\Psi$ of $\Psi_0$ to $\cl K(H)$ is then a completely positive 
extension of $S_{\nph}|_{\cl S_0(\kappa)}$. 
Let $\Psi^{**}$ be the second dual of $\Psi$, and 
$\cl E : \cl B(H\otimes K)^{**}\to \cl B(H\otimes K)$ be the canonical projection. 
We have that the map $\tilde{\Psi} = \cl E\circ \Psi^{**} : \cl B(H) \to \cl B(H\otimes K)$ is completely positive and weak* 
continuous extension of $S_{\nph}$. 
Let $\hat{\Psi} : \cl B(H)\otimes 1 \to \cl B(H\otimes K)$ 
(resp. $\hat{\Phi} : \cl S(\kappa)\otimes 1 \to \cl B(H\otimes K)$)
be the map given by 
$\hat{\Psi}(T\otimes I) = \tilde{\Psi}(T)$ (resp. $\hat{\Phi}(T\otimes I) = S_{\nph}(T)$); then
$\hat{\Psi}$ is a completely positive extension of map $\hat{\Phi}$.
Note that $\hat{\Phi}$ is a $\cl D \otimes 1$-bimodule map.
By~\cite[Exercise 7.4]{Pa}, $\hat{\Psi}$ is a $\cl D \otimes 1$-bimodule map. 
By Remark~\ref{r_modc0}, 
there exists $\psi\in \frak{S}(X,K)$ such that 
$\tilde{\Psi} = S_{\psi}$; the function $\psi$ is the desired positive extension of $\nph$. 
\end{proof}

If $\cl S$ is an operator system,
we write $\cl S^{++}$ for the cone of all positive finite rank operators in $\cl S$. 
If $\cl T$ is an operator system, we call a linear map $\Phi : \cl S\to \cl T$ \emph{strictly positive}
if $\Phi(S)\in \cl T^+$ whenever $S\in \cl S^{++}$. 
We call $\Phi$ \emph{strictly completely positive} if $\Phi^{(n)}$ is strictly positive for all $n\in \bb{N}$. 
A Schur multiplier $\nph : \kappa \to \cl B(K)$ will be called strictly positive (resp. strictly completely positive) 
if the map $S_{\nph} : \cl S(\kappa) \to \cl B(H\otimes K)$ is 
strictly positive (resp. strictly completely positive).

\begin{lemma}\label{l_de}
Let $\kappa$ be a positivity domain. 
Every positive finite rank operator in $M_n(\cl S(\kappa))$ 
has the form $(T_{k_{i,j}})_{i,j=1}^n$, where $k_{i,j}\in L^2(\kappa)$, $i,j = 1,\dots,n$.
\end{lemma}

\begin{proof}
Recall that $\cl S_2(\kappa) = \{T_k : k\in L^2(\kappa)\}$ and $\cl S_0(\kappa) = \overline{\cl S_2(\kappa)}^{\|\cdot\|}$.
It follows that $M_n(\cl S_0(\kappa)) = \overline{M_n(\cl S_2(\kappa))}^{\|\cdot\|}$.
Suppose that $T \in M_n(\cl S(\kappa))^{++}$ and let $T = (T_{i,j})_{i,j=1}^n$, where $T_{i,j}\in \cl S(\kappa)$,
$i,j=1,\dots,n$. 
Since $T$ has finite rank, so does $T_{i,j}$; in particular, $T_{i,j}$ is a Hilbert-Schmidt operator 
and, by \cite[Lemma 6.1]{eks}, $T_{i,j} \in \cl S_2(\kappa)$. 
\end{proof}

Recall that the Banach space projective tensor product 
\[\cl T(X) = L^2(X,\mu) \hat{\otimes } L^2(X,\mu)\]
can be canonically identified
with the predual of $\cl B(H)$ (and the dual of $\cl K(H)$). Indeed, each element $h\in
\cl T(X)$ can be written as a series $h = \sum_{i=1}^{\infty} f_i\otimes g_i$, 
where $\sum_{i=1}^{\infty} \|f_i\|_2^2 < \infty$
and $\sum_{i=1}^{\infty} \|g_i\|_2^2 < \infty$, and the pairing is then given
by 
\[\langle T,h\rangle = \sum_{i=1}^{\infty} (Tf_i,\overline{g_i}), \ \ \ T\in \cl B(H).\] 
We have~\cite{a} that $h$ can be identified with a complex function
on $X\times X$, defined up to a marginally null set, and given by
$$h(x,y) = \sum_{i=1}^{\infty} f_i(x)g_i(y).$$
The positive cone $\cl T(X)^+$ consists, by definition, of all 
functions $h\in \cl T(X)$ that give rise to positive functionals on $\cl B(H)$,
that is, functions $h$ of the form 
$h = \sum_{i=1}^{\infty} f_i\otimes \overline{f_i}$,  where $\sum_{i=1}^{\infty} \|f_i\|_2^2 < \infty$.
It is well-known that a function
$\nph \in L^{\infty}(X\times X)$ is a Schur multiplier
if and only if, for every $h\in \cl T(X)$, there exists $h'\in \cl T(X)$ such that 
$\nph h \sim h'$ (see \cite{peller}). 
In particular, if the measure $\mu$ is finite then $\frak{S}(X,\bb{C})$ can be naturally identified with 
a subspace of $\cl T(X)$.

\begin{theorem}\label{th_opsysext}
Let $\kappa\subseteq X\times X$ be a positivity domain. The following are equivalent:

(i) \ for every separable Hilbert space $K$, 
every strictly positive Schur multiplier $\nph : \kappa\to \cl B(K)$ is 
strictly completely positive;

(ii) for every $n\in \bb{N}$, every positive finite rank operator in $M_n(\cl S(\kappa))$ is the norm limit of 
sums of operators of the form $(D_i S D_j^*)_{i,j}$, where $(D_i)_{i=1}^n\subseteq \cl D$ and $S\in \cl S(\kappa)^{++}$. 
\end{theorem}
\begin{proof}
(i)$\Rightarrow$(ii) 
We first assume that the measure $\mu$ is finite. 
Suppose that there exists $n\in \bb{N}$ and a positive finite rank operator $T\in M_n(\cl S(\kappa))$ 
that is not equal to the limit, in the 
norm topology, of the operators of the form $(D_iS D_j^*)_{i,j=1}^n$, 
where $(D_i)_{i=1}^n\subseteq \cl D$ and $S\in \cl S(\kappa)^{++}$.
By Lemma \ref{l_de}, $T = (T_{k_{i,j}})_{i,j=1}^n$, for some 
$k_{i,j}\in L^2(\kappa)$, $i,j = 1,\dots,n$.
By a geometric form of Hahn-Banach's Theorem, there exist a 
norm continuous functional $\omega : M_n(\cl S_0(\kappa)) \to \bb{C}$
and $\gamma < 0$ such that 
\begin{equation}\label{eq_omnew}
\omega(T) < \gamma \ \mbox{ and } \ 
\omega\left((D_iS D_j^*)_{i,j=1}^n\right) \geq 0, \ \  S\in \cl S(\kappa)^{++}, (D_i)_{i=1}^n\subseteq \cl D.
\end{equation}
Let $\omega_{i,j} : \cl S_0(\kappa) \to \bb{C}$ be the norm continuous functionals
such that 
$$\omega((S_{i,j})_{i,j = 1}^n) = \sum_{i,j = 1}^n \omega_{i,j}(S_{i,j}), \ \ \  S_{i,j}\in \cl S_0(\kappa), \ i, j = 1,\dots,n.$$
After extending $\omega_{i,j}$ to $\cl K(H)$, we may assume that $\omega_{i,j}\in \cl T(X)$ for $i, j = 1, \dots, n$.

Suppose first that $\omega_{i,j}\in \frak{S}(X,\bb{C})$, $i,j = 1,\dots,n$.
Identify $\omega$ with the function (denoted by the same symbol) 
$\omega : X\times X\to M_n$, given by $\omega(x,y) = (\omega_{i,j}(x,y))_{i, j=1}^n$. 
Since $S_{\omega} : \cl S_2(H)\to \cl B(H)\otimes M_n$ is given by 
$S_{\omega}(T_k) = (S_{\omega_{i,j}}(T_k))$, $k\in L^2(X\times X)$, and the maps $S_{\omega_{i,j}}$ are
completely bounded, we have that the map $S_{\omega}$ is completely bounded, that is, 
$\omega\in \frak{S}(X,M_n)$. 

We claim that $S_{\omega}^{(n)}$ is not strictly positive. 
%By Proposition \ref{p_cpext}, it suffices to show that $S_{\omega}$ is not completely positive. 
Note that 
$$S_{\omega}^{(n)}(T) = \left(S_{\omega_{i,j}}(T_{k_{p,q}})\right)_{i,j,p,q}.$$
Writing $e$ for the vector in $H^n$ with all its entries equal to the constant function $1$, we have that 
\begin{eqnarray}\label{eq_con}
\gamma  & > & 
\omega(T) = \sum_{i,j=1}^n \int_{\kappa} \omega_{i,j}(x,y)k_{i,j}(x,y) d(\mu\times \mu)(x,y)\nonumber\\
& = & 
\left(\left(S_{\omega_{i,j}}(T_{k_{i,j}})\right)_{i,j}e,e\right).
\end{eqnarray}
Suppose that $S_{\omega}^{(n)}(T)$ is positive. Then its submatrix 
$(S_{\omega_{i,j}}(T_{k_{i,j}}))_{i,j}$ is positive, which contradicts (\ref{eq_con}).
%It follows that $\omega$ does not have a positive extension.

We now show that $S_{\omega}$ is strictly positive. 
Let $S\in \cl S(\kappa)^{++}$. 
Using Lemma \ref{l_de}, write $S = T_k$ for some $k\in L^2(\kappa)$. 
We have that $S_{\omega}(S) = (T_{\omega_{i,j} k})_{i, j=1}^n$. 
For $i = 1,\dots,n$, let $\xi_i\in L^{\infty}(X,\mu)$ and note that, since $\mu$ is finite, $\xi_i\in H$.
Let $D_i = M_{\xi_i}$, $i = 1, \dots, n$, and set $\xi = (\xi_i)_{i=1}^n$. 
We have that 
\begin{eqnarray*}
\left(S_{\omega}(S)\xi,\xi\right) 
& = & 
\sum_{i, j=1}^n (T_{\omega_{i,j} k}\xi_j,\xi_i)\\
& = & 
\sum_{i, j=1}^n \int_{\kappa} \omega_{i,j}(x,y) k(x,y)\xi_j(x)\overline{\xi_i(y)} d(\mu\times\mu)(x,y)\\
& = & \omega\left((D_i^*SD_j)_{i,j=1}^n\right) \geq 0.
\end{eqnarray*}
Since $L^{\infty}(X,\mu)$ is dense in $H$, we have that 
$S_{\omega}(S)\in M_n(\cl B(H))^+$. 

Now relax the assumption that $\omega_{i,j} \in \frak{S}(X,\bb{C})$.
By standard arguments (see e.g. the proof of \cite[Lemma 3.13]{akt}), 
there exist measurable sets $X_m\subseteq X$ with $X_m\subseteq X_{m+1}$, $m\in \bb{N}$, such that 
$\mu(X\setminus X_m) \to_{m\to\infty} 0$ and the restriction $\omega^{(m)}_{i,j}$ of $\omega_{i,j}$ to $X_m \times X_m$
belongs to $\frak{S}(X_m,\bb{C})$ for all $m\in \bb{N}$.
Let $\omega^{(m)} : X\times X \to M_n$ be the function given by $\omega^{(m)}(x,y) = (\omega_{i,j}^{(m)}(x,y))_{i,j}$
if $(x,y)\in X_m\times X_m$ and $\omega^{(m)}(x,y) = 0$ otherwise,
and note that $\omega^{(m)}$ defines a functional on $M_n(\cl K(H))$ in the natural way
(which will be denoted by the same symbol). 
Let $P_m$ be the projection from $H$ onto $L^2(X_m)$. 
We have that 
$$\omega^{(m)}(R) = \omega((P_m\otimes I_n) R (P_m\otimes I_n)), \ \ \ R\in M_n(\cl K(H)).$$
Since $(P_m\otimes I_n) R (P_m\otimes I_n) \to_{m\to \infty} R$ in norm, for every $R\in M_n(\cl K(H))$, 
we have that 
(\ref{eq_omnew}) eventually holds true for $\omega^{(m)}$ in the place of $\omega$.
By the previous paragraph, $\omega^{(m)}$ is a Schur multiplier 
for which $S_{\omega^{(m)}}$ is strictly positive, but not strictly completely positive. 

Finally, relax the assumption that $\mu$ be finite. 
Let $(X_m)_{m\in \bb{N}}$ be an increasing sequence of sets of finite measure such that $\cup_{m=1}^{\infty} X_m = X$, 
and let $Q_m$ be the projection from $H$ onto $L^2(X_m)$, $m\in \bb{N}$. 
Let $T\in M_n(\cl S(\kappa))^{++}$. Since $T$ is a positive operator of finite rank, 
$(Q_mTQ_m)_{m\in \bb{N}}$
is a sequence of positive finite rank operators, converging to $T$ in norm. By the first part of the proof,
$Q_mTQ_m$ is a norm limit of 
operators of the form $(D_i S D_j^*)_{i,j}$, where $(D_i)_{i=1}^n\subseteq \cl D$ and $S\in \cl S(\kappa)^{++}$.
The conclusion follows.

(ii)$\Rightarrow$(i) 
Let $\nph : \kappa\to \cl B(K)$ be a Schur multiplier such that $S_{\nph} : \cl S(\kappa)\to \cl B(H\otimes K)$
is strictly positive. 
It follows from the assumption and fact that $S_{\nph}$ is a $\cl D$-bimodule map that 
$S_{\nph}^{(n)}(T)$ is positive whenever $T\in M_n(\cl S(\kappa))^{++}$. 
\end{proof}

\begin{definition}
Let $\kappa$ be a positivity domain. We call $\kappa$ \emph{rich} if
$$M_n(\cl S(\kappa))^{+} = \overline{M_n(\cl S(\kappa))^{++}}^{w^*} \ \ \mbox{ for every } n\in \bb{N}.$$
\end{definition}

Suppose that $X$ is a countable set equipped with counting measure. 
In this case, positivity domains can be identified with undirected graphs 
with vertex set $X$ in the natural way. This identification will be made in the subsequent remark and in 
Theorem \ref{c_ch}.

\begin{remark}\label{r_disr}
Let $X$ be a countable set. Then any graph $\kappa\subseteq X \times X$ is rich. 
\end{remark}

\begin{proof}
For $X = \bb{N}$, write $Q_m$ for the projection onto the span of $\{e_i\}_{i=1}^m$, $m\in \bb{N}$, 
where $\{e_i\}_{i\in \bb{N}}$ is the standard basis of $\ell^2$.
If $T\in M_n(\cl S(\kappa))^{+}$ then $((Q_m\otimes I_n) T (Q_m\otimes I_n))_{m\in \bb{N}}$ is a sequence
in $M_n(\cl S_2(\kappa))^{++}$, converging in the weak* topology to $T$. 
\end{proof}

By Proposition \ref{p_cpext}, if a Schur multiplier $\nph : \kappa\to \cl B(K)$ has a positive extension then 
the map $S_{\nph} : \cl S(\kappa) \to \cl B(H\otimes K)$ is necessarily positive. We call
$\nph$ \emph{admissible} 
if $S_{\nph}$ is a positive map.
The main result of this section is a characterisation of when an admissible 
Schur multiplier has a positive extension, in terms of the maximal operator $\cl D$-system structure 
defined in Section \ref{s_deoass}.
Note that $\cl S(\kappa)$ is a dual AOU $\cl D$-space in the natural fashion.

\begin{theorem}\label{th_rich}
Let $\kappa\subseteq X\times X$ be a rich positivity domain. The following are equivalent:

(i) \ for every separable Hilbert space $K$, every admissible Schur multiplier $\nph : \kappa \to \cl B(K)$ 
has a positive extension;

(ii) $\cl S(\kappa) = \omax_{\cl D}^{w^*}(\cl S(\kappa))$.
\end{theorem}

\begin{proof}
(i)$\Rightarrow$(ii) 
Let $\nph : \kappa \to \cl B(K)$ be a strictly positive Schur multiplier. 
Since $\cl S(\kappa)^+ = \overline{\cl S(\kappa)^{++}}^{w^*}$ and $S_{\nph}$ is weak* continuous, 
$S_{\nph}$ is positive.
By the assumption and Proposition \ref{p_cpext}, $S_{\nph}$ is completely positive. 
In particular, $S_{\nph}$ is strictly completely positive. 
By Theorem \ref{th_opsysext} and the fact that the matricial cones of any operator system are 
norm closed, we have that 
\begin{equation}\label{eq_mn++}
M_n(\cl S(\kappa))^{++} \subseteq M_n(\omax\mbox{}_{\cl D}(\cl S(\kappa)))^+.
\end{equation}
Since $\kappa$ is rich, by taking weak* closures on both sides in (\ref{eq_mn++}) 
we obtain that 
\begin{equation}\label{eq_mn+}
M_n(\cl S(\kappa))^{+} \subseteq M_n(\omax\mbox{}_{\cl D}^{w^*}(\cl S(\kappa)))^+.
\end{equation}
Since the converse inclusion in (\ref{eq_mn+}) always holds, we conclude that 
$\cl S(\kappa) = \omax_{\cl D}^{w^*}(\cl S(\kappa))$.

(ii)$\Rightarrow$(i) follows from Theorem \ref{th_nacp} and Proposition \ref{p_cpext}.
\end{proof}

Theorem \ref{th_rich} and Remark \ref{r_disr} have the following immediate corollary. 
In the case where $X$ is finite, it is a reformulation, in terms of operator system structures,
of \cite[Theorem 4.6]{pps}.

\begin{corollary}\label{c_disc}
Let $X$ be a countable set, equipped with counting measure and $\kappa\subseteq X\times X$ be a 
symmetric set containing the diagonal. The following are equivalent:

%\begin{itemize}
(i) \ for every Hilbert space $K$, every admissible Schur multiplier $\nph : \kappa\to \cl B(K)$
has a positive extension;

(ii) $\cl S(\kappa) = \omax_{\cl D}^{w^*}(\cl S(\kappa))$.
%\end{itemize}
\end{corollary}

Let $X$ be a countable set. Recall that a graph $\kappa\subseteq X\times X$ is called chordal 
if every 4-cycle in $\kappa$ has an edge connecting two non-consecutive vertices of the cycle 
(see e.g. \cite{pps}).

\begin{theorem}\label{c_ch}
Let $X$ be a countable set and $\kappa\subseteq X\times X$ be a chordal graph. 
Then $\cl S(\kappa) = \omax_{\cl D}^{w^*}(\cl S(\kappa))$.
\end{theorem}

\begin{proof}
Fix $n\in \bb{N}$ and let $[n] = \{1,\dots,n\}$.
Suppose that $\kappa \subseteq X\times X$ is a chordal graph. 
Let 
$$\kappa^{(n)} = \left\{((x,i),(y,j)) \in \left(X\times [n]\right)\times \left(X\times [n]\right) : (x,y)\in \kappa\right\}.$$
Then $\kappa^{(n)}$ is a chordal graph on $X\times [n]$. 
By \cite[Theorem 2.5]{llt}, every positive operator in $M_n(\cl S(\kappa))$ is a weak* limit 
of rank one positive operators in $M_n(\cl S(\kappa))$. 

Suppose that $K$ is a Hilbert space and $\nph : \kappa \to \cl B(K)$ is a Schur multiplier such that $S_{\nph} : \cl S(\kappa)\to \cl B(H \otimes K)$
is a positive map. 
Let $R\in M_n(\cl S(\kappa))$ be a positive rank one operator. 
After identifying $M_n(\cl S(\kappa))$ with $\cl S(\kappa^{(n)})$, 
we see that there exists a subset $\alpha \subseteq X\times [n]$ such that 
$R$ is supported on $\alpha\times\alpha$. 
Let 
$$\beta = \{x\in X : \exists \ i\in [n] \mbox{ with } (x,i)\in \alpha\}.$$
Since $\alpha\times\alpha \subseteq \kappa^{(n)}$, we have that $\beta\times\beta \subseteq \kappa$.
Setting $\tilde{\beta} = \beta\times [n]$, we have that
$\alpha\subseteq \tilde{\beta}$, and hence $R$ is supported on 
$\tilde{\beta} \times \tilde{\beta}$. 
The restriction $\psi$ of $\nph$ to $\beta\times\beta$ is a positive Schur 
multiplier. By Theorem \ref{th_modc}, the map $S_{\psi} : \cl S(\beta\times\beta)\to \cl B(H\otimes K)$ is
completely positive.
Thus, $S_{\nph}^{(n)}(R) = S_{\psi}^{(n)}(R) \in \cl B(H\otimes K)^+$.
Since $S_{\nph}$ is weak* continuous, the previous paragraph implies that $S_{\nph}$ is completely positive.
By Proposition \ref{p_cpext}, $\nph$ has a positive extension and, by Corollary \ref{c_disc}, 
$\cl S(\kappa) = \omax_{\cl D}^{w^*}(\cl S(\kappa))$.
\end{proof}

%%%%%%%%%%%%%%%%%%%%%%%%%%%%%%%%%%%%%%%%%%%%%%%%%%%%%%%%%%%%%%%%%%%
%%%%%%%%%%%%%%%%%%%%%%%%%%%%%%%%%%%%%%%%%%%%%%%%%%%%%%%%%%%%%%%%%%%%%%%%%%%%%%%%%%%%%%%%%%%%%%%%%%%%%%%%%%%%%%%%%%%%%%%%%%%%%%%%%%%%%%%%%%%%%%%%%%%%%%%%%%%%%%%%%%%%%%%%%%%%%%%%%%%%%%%%%%%%%%%%%%%%%%%%%

\end{document}